\documentclass[a4paper]{scrartcl}
\usepackage[utf8]{inputenc}
\usepackage[english]{babel}
\usepackage{xcolor}
\usepackage{graphicx}
\usepackage{hyperref}
\usepackage[T1]{fontenc}
\usepackage{lmodern}
\usepackage{rotating}
\usepackage{tikz}
\usepackage{tikz-cd}
\usepackage[backend=biber,sortlocale=de_DE,url=false,doi=false,isbn=false,giveninits=true,style=alphabetic,maxbibnames=1000,date=year,sorting=nyt]{biblatex}
\bibliography{../../jabref.bib}
\DeclareFieldFormat*{title}{\mkbibemph{#1\isdot}}
\DeclareFieldFormat*{journaltitle}{#1\isdot}
\renewbibmacro{in:}{}
\usepackage{amsmath}
\usepackage{amsfonts}
\usepackage{amssymb}
\usepackage{amsthm}
\usepackage{mathtools}
\usepackage{stmaryrd}
\usepackage{enumerate}
\usepackage{verbatim}
\usepackage{dsfont}
\usepackage{mathabx}
\usepackage{faktor}
\usepackage{lipsum}
\usepackage{CJKutf8}
\DeclareMathOperator\defect{def}

\DeclareMathOperator\diag{diag}

\DeclareMathOperator\tr{tr}

\DeclareMathOperator\wt{wt}

\DeclareMathOperator\DBG{DBG}

\DeclareMathOperator\LP{LP}

\DeclareMathOperator\paths{paths}

\DeclareMathOperator\codim{codim}
\DeclareMathOperator\sitype{type}
\DeclareMathOperator\wts{wts}
\def\GL{{\mathrm{GL}}}

\def\dom{{\mathrm{dom}}}

\hypersetup{colorlinks=false, linkbordercolor={white}, hidelinks, pdfauthor={Felix Schremmer},pdftitle={Affine Deligne--Lusztig varieties via the double Bruhat graph I: Semi-infinite orbits}}
\usepackage{csquotes}
\author{Felix Schremmer}\date{\today}
\title{Affine Deligne--Lusztig varieties via the double Bruhat graph I:\\Semi-infinite orbits}
\numberwithin{equation}{section}
\newtheorem{theorem}[equation]{Theorem}
\newtheorem{proposition}[equation]{Proposition}
\newtheorem{lemma}[equation]{Lemma}
\newtheorem{corollary}[equation]{Corollary}

\theoremstyle{definition}
\newtheorem{definition}[equation]{Definition}

\theoremstyle{remark}
\newtheorem{example}[equation]{Example}
\newtheorem{remark}[equation]{Remark}
\def\abs#1{{\left\lvert{#1}\right\rvert}}
\def\doubleparen#1{{(\!({#1})\!)}}
\def\doublebrack#1{{[\![{#1}]\!]}}

\let\oldqedsymbol\qedsymbol
\def\qedaddendum{}
\def\qedsymbol{\oldqedsymbol\qedaddendum}

\def\af{{\mathrm{af}}}
\def\rightqed{\pushQED{\qed}\qedhere\popQED}

\begin{document}
\maketitle
\begin{abstract}
We introduce a new language to describe the geometry of affine Deligne--Lusztig varieties in affine flag varieties. This first part of a two paper series develops the definition and fundamental properties of the double Bruhat graph by studying semi-infinite orbits. This double Bruhat graph was originally introduced by Naito--Watanabe to study periodic $R$-polynomials. We use it to describe the geometry of many affine Deligne--Lusztig varieties, overcoming a previously ubiquitous regularity condition.
\end{abstract}
\section{Introduction}
Shimura varieties play a central role in the Langlands program. By giving the Shimura variety an interpretation as a moduli space (e.g.\ of certain abelian varieties), one obtains an integral model whose generic fibre recovers the original Shimura variety \cite{Rapoport2005, Kisin2018, Pappas2023}. The special fibre of such an integral model is then known as the mod $p$ reduction of the Shimura variety. In the case of a parahoric level structure, the geometry of each special fibre is closely related to the geometry of corresponding affine Deligne--Lusztig varieties. Similar affine Deligne--Lusztig varieties occur in the special fibres of moduli spaces of local $G$-shtukas \cite{Viehmann2018}.

We consider a reductive group $G$ defined over a local field $F$, whose completion of the maximal unramified extension we denote by $\breve F$. Given a parahoric subgroup $K\subset G(\breve F)$, we associate the affine Deligne--Lusztig variety $X_x^K(b)$ to any two elements $x$ and $b$ in $G(\breve F)$ \cite[Definition~4.1]{Rapoport2005}. It is defined as locally closed subvariety of the partial flag variety associated with $K$. It has the structure of a finite-dimensional scheme or perfect scheme over the residue field $k$ of $\breve F$, whose geometric points are given by
\begin{align*}
X_x(b) = X_x^K(b) = \{g\in G(\breve F)/K\mid g^{-1} b \sigma(g)\in KxK\}\subset G(\breve F)/K.
\end{align*}
Here, $\sigma$ denotes the Frobenius of $\breve F/F$. One notes that the affine Deligne--Lusztig variety depends, up to isomorphism, only on the double coset $KxK\subset G(\breve F)$ and the $\sigma$-conjugacy class $[b] = \{g^{-1}b\sigma(g)\mid g\in G(\breve F)\}$. The intersection of these two sets $KxK\cap [b]$ is known as Newton stratum and its geometry is closely related to that of the affine Deligne--Lusztig variety. The most important questions, in increasing order of difficulty, are the following:
\begin{enumerate}[(Q1)]
\item When is $X_x(b)$ empty? Equivalently, when is the Newton stratum empty?
\item If $X_x(b)\neq\emptyset$, what is its dimension?
\item How many irreducible components of any given dimension does $X_x(b)$ have?
\end{enumerate}
The final question is especially interesting when the given dimension is equal to $\dim X_x(b)$, i.e.\ if one asks for the top dimensional irreducible components. The number of such irreducible components will in general be infinite. However, the $\sigma$-centralizer of $b$,
\begin{align*}
J_b(F) = \{g\in G(\breve F)\mid g^{-1}b\sigma(g) = b\},
\end{align*}
acts by left multiplication on $X_x(b)$. There are only finitely many orbits of irreducible components up to the $J_b(F)$-action, which is how (Q3) should be understood. Equivalently, one may ask for the number of top dimensional irreducible components of the Newton stratum.

The first step towards answering these three questions is to give a suitable parametrizations for the double cosets \begin{align*}K\setminus G(\breve F)/K = \{KxK\mid x\in G(\breve F)\}\end{align*} and $\sigma$-conjugacy classes \begin{align*}B(G) = \{[b]\mid b\in G(\breve F)\}.\end{align*}

We will assume that the group $G$ is split and choose a split maximal torus $T$. For this introduction, this merely provides a slightly more convenient notation. More importantly, this restriction is essential for the remainder of the article due to the dependence on the earlier work \cite{Goertz2006} with this assumption.

We first consider the case of a hyperspecial subgroup $K$. If $G$ is already defined over the ring of integers $\mathcal O_{\breve F}$ of $\breve F$, then $K=G(\mathcal O_{\breve F})$ would be a typical example of this. For hyperspecial $K$, the double cosets $K\setminus G(\breve F)/K$ are parametrized by the dominant elements of the cocharacter lattice $X_\ast(T)$. Explicitly, evaluation at a uniformizer assigns to each cocharacter $\mu\in X_\ast(T)$ a representative $\dot \mu\in T(\breve F)$, and then each double coset $KxK$ contains the representative of precisely one dominant cocharacter $\mu$. We also write $K\mu K$ for $K\dot\mu K$. Then the \emph{Cartan decomposition} is given by
\begin{align*}
G(\breve F) = \bigsqcup_{\substack{\mu \in X_\ast(T)\\\text{dominant}}} K\mu K.
\end{align*}

If $K=I$ is an Iwahori subgroup, the double cosets $I\setminus G(\breve F)/I$ are parametrized by the extended affine Weyl group $\widetilde W$. This group can be defined as $N_G(T)(\breve F)/T(\mathcal O_{\breve F})$ and it is isomorphic to the semidirect product of the Weyl group $W = N_G(T)(\breve F)/T(\breve F)$ of $G$ with the cocharacter lattice $X_\ast(T)$. Here, we write $N_G(T)$ for the normalizer of $T$ inside $G$. Choosing for each $x\in\widetilde W$ a representative $\dot x \in N_G(T)(\breve F)$, the double coset $Ix I = I\dot x I$ is independent of this choice. We obtain the \emph{Iwahori--Bruhat decomposition}
\begin{align*}
G(\breve F) = \bigsqcup_{x\in\widetilde W} IxI.
\end{align*}

For general parahoric levels, one may parametrize $K\setminus G(\breve F)/K$ by suitable double cosets in $\widetilde W$, but we will not consider this case.

The $\sigma$-conjugacy class of an element $b\in G(\breve F)$ is uniquely determined by two invariants; this is a celebrated result of Kottwitz \cite{Kottwitz1985, Kottwitz1997}. These invariants are known as the (dominant) Newton point $\nu(b) \in X_\ast(T)\otimes\mathbb Q$ and the Kottwitz point $\kappa(b) \in \pi_1(G)$. Here, $\pi_1(G) = X_\ast(T)/\mathbb Z\Phi^\vee$ is the Borovoi fundamental group and $\mathbb Z\Phi^\vee$ is the coroot lattice. 
Following He \cite[Theorem~3.7]{He2014}, one may also parametrize the set $B(G)$ using $\sigma$-conjugacy classes in $\widetilde W$.

Since the Kottwitz point $\kappa : G(\breve F)\rightarrow \pi_1(G)$ parametrizes the connected components of the partial flag variety, we get that $\kappa(x) = \kappa(b)$ is a necessary condition for $X_x(b)\neq\emptyset$. Once this condition is imposed, we may focus on comparing the above parametrization for $x$ with the Newton point $\nu(b)$.

In the case of hyperspecial level, our three initial questions have been mostly solved after concentrated effort by many researchers. For the split case under consideration, we can summarize the results as follows (while still providing references for the general case).
\begin{theorem}\label{thm:hyperspecial}
Assume that $K$ is hyperspecial. Let $[b]\in B(G)$ and $\mu\in X_\ast(T)$ be a dominant coweight.
\begin{enumerate}[(a)]
\item  The affine Deligne--Lusztig variety $X_\mu(b)$ is non-empty if and only if the \emph{Mazur inequality} is satisfied: That is, $\kappa(b) = \kappa(\mu)$ and $\nu(b)\leq \mu$ in the dominance order of $X_\ast(T)\otimes\mathbb Q$. Conjectured by Kottwitz--Rapoport, proved by \cite{Rapoport1996, Gashi2010, He2014}.
\item If $X_\mu(b)\neq\emptyset$, it is equidimensional of dimension
\begin{align*}
\dim X_\mu(b) = \frac 12\left(\langle \mu - \nu(b),2\rho\rangle - \defect(b)\right).
\end{align*}
Here, $\defect(b)$ denotes the \emph{defect} of $b$, which is defined as $\mathrm{rk}_F(G) - \mathrm{rk}_F(J_b)$, cf.\ \cite{Chai2000, Kottwitz2006}. Conjectured by Rapoport, proved by \cite{Goertz2006, Viehmann2006, Hamacher2015, Takaya2022}.
\item  The $J_b(F)$-orbits of irreducible components of $X_\mu(b)$ are in bijection to a certain basis of the weight space $M_\mu(\lambda(b))$ of the irreducible quotient $M_\mu$ of the highest weight Verma module $V_\mu$. Here, $\lambda(b)\in X_\ast(T)$ is the largest cocharacter satisfying $\lambda(b)\leq \nu(b)$ and $\kappa(\lambda(b)) = \kappa(b)$. Conjectured by Chen--Zhu, proved by \cite{Zhou2020, Nie2022}.
\end{enumerate}
\end{theorem}
We see that once $\kappa(\mu) = \kappa(b)$ is required, the difference $\mu-\nu(b)$ resp.\ $\mu-\lambda(b)$ determines most properties of $X_\mu(b)$, using e.g.\ the fact that the dimension of the weight space $M_\mu(\lambda(b))$ can be approximated using the dimension of $V_\mu(\lambda(b))$, which is \emph{Kostant's partition function} applied to the difference $\mu-\lambda(b)$. Under certain regularity conditions, the dimensions of the two weight spaces will be equal.

Let us now summarize the most important results known in the case of Iwahori level structure. Assume that $K=I$ is an Iwahori subgroup. Pick an element $x\in\widetilde W\cong W\ltimes X_\ast(T)$ and write it as $x = wt^\mu$ where $w\in W$ and $\mu\in X_\ast(T)$. The element $t\in F$ is the uniformizer, so the representative of $t^\mu$ in $G$ is given by the image of $t$ under the cocharacter $\mu$. We set
\begin{align*}
B(G)_x = \{[b]\in B(G)\mid IxI\cap [b]\neq \emptyset\} = \{[b]\in B(G)\mid X_x(b)\neq\emptyset\}.
\end{align*}
It should not be surprising that $B(G)_x$ contains a unique minimal and a unique maximal element, and both have been explicitly described \cite{Viehmann2014, Viehmann2021, Schremmer2022_newton}.

For any $[b]\in B(G)_x$, we know that $\dim X_x(b)\leq d_x(b)$ \cite[Theorem~2.30]{He2015}, where $d_x(b)$ is the virtual dimension defined by He \cite[Section~10]{He2014}. It is defined as
\begin{align*}
d_x(b) = \frac 12\left(\ell(x)+\ell(\eta(x))-\langle \nu(b),2\rho\rangle-\defect(b)\right).
\end{align*}
The definition of $\eta(x)\in W$ is somewhat technical, so we will not recall it here. A striking feature of the virtual dimension is that it is a simple sum of four terms, the first two only depending on $x\in\widetilde W$ and the latter two only depending on $[b]\in B(G)$. The virtual dimension behaves best when the element $x$ satisfies a certain regularity condition known as being in a shrunken Weyl chamber, cf.\ \cite[Example~2.8]{Schremmer2022_newton}.
\begin{theorem}
\label{thm:virtDim}
Let $x \in \widetilde W$, denote the largest element in $B(G)_x$ by $[b_x]$ and the smallest one by $[m_x]$.
\begin{enumerate}[(a)]
\item Suppose that $\dim X_x(b_x) = d_x(b_x)$. Then
\begin{align*}
B(G)_x = \{[b]\in B(G)\mid [m_x]\leq [b]\leq [b_x]\}.
\end{align*}
For each $[b]\in B(G)_x$, the variety $X_x(b)$ is equidimensional of dimension $\dim X_x(b) = d_x(b)$ \cite[Theorem~1.1]{Milicevic2020}. The elements $x$ satisfying this condition have been classified, cf.\ \cite[Theorem~1.2]{Schremmer2022_newton}.
\item Suppose that $x$ lies in a shrunken Weyl chamber. Then $\dim X_x(m_x) = d_x(m_x)$ \cite[Theorem~1.1~(2)]{Viehmann2021}. For each $[m_x]\leq [b]\in B(G)$ with $\nu(b_x) - \nu(b)$ sufficiently large, we have $X_x(b)\neq \emptyset$ and $\dim X_x(b) = d_x(b)$ \cite[Theorem~1.1]{He2021a}.
\end{enumerate}
\end{theorem}
While in a quantitative sense \enquote{most} elements of $\widetilde W$ lie in a shrunken Weyl chamber, the examples coming from Shimura varieties typically do not. In fact, the difference between virtual dimension and dimension for basic $[b]$ can be quite large for these examples.

It should not surprise that there are many examples where dimension and virtual dimension differ. For non-shrunken elements, the notion of virtual dimension behaves poorly. E.g.\ it is not compatible with certain natural automorphisms of the reductive group $G$ that preserve the Iwahori subgroup $I$ (hence induce isomorphisms of affine Deligne--Lusztig varieties). Even for shrunken elements, we expect to have $\dim X_x(b) = d_x(b)$ only for \enquote{small} elements $[b]\in B(G)_x$.

The case (a) in Theorem~\ref{thm:virtDim} is known as the \emph{cordial} case. While (Q1) and (Q2) have \enquote{ideal} answers in this case, these descriptions are too good to be true in general. It is known that the set $B(G)_x$ will in general contain gaps and that affine Deligne--Lusztig varieties may fail to be equidimensional. If $x$ is cordial, the answer to (Q3) does not seem to be known in general.

We may summarize that (Q1) and (Q2) are well understood if $[b]\in B(G)$ is small relative to $x$ and $x$ is in a shrunken Weyl chamber, or if $x$ enjoys some exceptionally good properties.

Moreover, all three questions are perfectly understood in case $[b] = [b_x]$ is the largest $\sigma$-conjugacy class in $B(G)_x$, also known as the generic $\sigma$-conjugacy class of $IxI$. We have $\dim X_x(b_x) = \ell(x)-\langle \nu(b_x),2\rho\rangle$ \cite[Theorem~2.23]{He2015}. Up to the $J_b(F)$-action, there is only one irreducible component in $X_x(b_x)$ \cite[Lemma~3.2]{Milicevic2020}. In order to describe $[b_x]$ in terms of $x$, one may (and arguably should) use the quantum Bruhat graph \cite{Milicevic2021, Schremmer2022_newton}.

The goal of this paper and its sequel is to introduce a new concept, which generalizes the virtual dimension in the case of Theorem \ref{thm:virtDim} (b) and also generalizes the known theory for the generic $\sigma$-conjugacy class. We give answers to all three above questions in many cases that were previously intractable.

In this article, we follow one of the oldest approaches towards affine Deligne--Lusztig varieties in the affine flag variety, namely the one developed by Görtz-Haines-Kottwitz-Reuman \cite[Section~6]{Goertz2006}. They consider the case of a split group $G$, an equal characteristic field $F$ and an integral $\sigma$-conjugacy class $[b]\in B(G)$; the element $x\in \widetilde W$ is allowed to be arbitrary. They compare the geometric properties of $X_x(b)$ (especially questions (Q1)--(Q3)) to similar geometric properties of intersections, in the affine flag variety, of $IxI$ with certain \emph{semi-infinite orbits}.

Given a Borel $B = TU$, we get another decomposition of $G(\breve F)$ resp.\ the affine flag variety:
\begin{align*}
G(\breve F) =\bigsqcup_{y\in \widetilde W} U(L)yI.
\end{align*}
The individual pieces $U(L)yI$ are called \emph{semi-infinite orbits}. Each Borel containing our fixed torus $T$ gives rise to a different decomposition. In our notation, we will fix $B$ and then consider the semi-infinite orbit decompositions associated with the conjugates $uBu^{-1} = \prescript u{}{}B$ for various $u\in W$.

In order to understand $X_x(b)$ following \cite[Theorem~6.3.1]{Goertz2006}, we have to understand the intersections
\begin{align}
IxI~\cap~\prescript u{}{}U(\breve F)yI~\subset~G(L)/I\label{eq:intro1}
\end{align}
for various $u\in W$ and $y\in\widetilde W\cap [b]$. One may ask questions (Q1)--(Q3) analogously for these intersections. Unfortunately, not many answers to these questions have been given in the previous literature, leaving basic geometric properties of \eqref{eq:intro1} largely open.
There is a decomposition of \eqref{eq:intro1} into subvarieties parametrized by folded alcove walks \cite[Theorem~7.1]{Parkinson2009}, which has been used to study affine Deligne--Lusztig varieties \cite{Milicevic2019}, but these results have often been difficult to apply in practice.

One may always find an element $v\in W$ such that $IxI\subseteq \prescript v{}{}U(L)xI$, and we will use this semi-infinite orbit to approximate $IxI$. Doing so (in the proof of Theorem~\ref{thm:generalizedMV} below), we can compare the intersection \eqref{eq:intro1} to the intersection
\begin{align}
\Bigl(\prescript v{}{}U(\breve F)\cap \prescript{uw_0}{}{}U(\breve F)\Bigr)xI~\cap~\prescript u{}{}U(\breve F)yI\label{eq:intro2}
\end{align}
for $v\in W$ such that $IxI\subseteq \prescript v{}{}U(\breve F)xI$. We write $w_0\in W$ for the longest element, so that $\prescript{uw_0}{}{}B$ is the Borel subgroup opposite to $\prescript u{}{}B$. As an application of our findings, we will later see in Proposition~\ref{prop:sioIntersections} that the large parentheses in \eqref{eq:intro2} are unnecessary, that is,
\begin{align*}
\Bigl(\prescript v{}{}U(\breve F)\cap \prescript{uw_0}{}{}U(\breve F)\Bigr)xI = 
\Bigl(\prescript v{}{}U(\breve F) xI\Bigr)\cap \Bigl(\prescript{uw_0}{}{}U(\breve F)xI\Bigr).
\end{align*}

The first part of this paper studies intersections of the form \eqref{eq:intro2}. This is a question of independent interest, whose answer we want to later apply to affine Deligne--Lusztig varieties.
A different motivation to study intersections as in \eqref{eq:intro2} is the following: One may naturally ask about the intersections of arbitrary semi-infinite orbits
\begin{align}
\prescript u{}{}U(\breve F)xI~\cap~\prescript v{}{}U(\breve F)yI~\subset~G(\breve F)/I.\label{eq:intro3}
\end{align}
Observe that the group $\prescript u{}{}U(\breve F)\cap \prescript v{}{}U(\breve F)$ acts by left multiplication on \eqref{eq:intro3}, and the orbits of this action will be infinite-dimensional (unless $u=vw_0$). However, each such orbit will contain a point of \eqref{eq:intro2}, so we may see \eqref{eq:intro2} as a finite-dimensional space of representatives of \eqref{eq:intro3}. Moreover, the intersection \eqref{eq:intro3} is empty if and only if the intersection \eqref{eq:intro2} is empty.

We study the intersection \eqref{eq:intro2} for arbitrary $x,y,u,v$ in Section~\ref{sec:semiInfiniteOrbits}. By comparing the valuation of root subgroups with the extended affine Weyl group, we get a decomposition of \eqref{eq:intro2} into finitely many locally closed subvarieties, each of them irreducible and finite dimensional.

It turns out that there is very convenient combinatorial tool to parametrize the subvarieties of this decomposition and to describe their dimensions. This is the double Bruhat graph, a combinatorial object introduced by Naito--Watanabe \cite[Section~5.1]{Naito2017} in order to study periodic $R$-polynomials. The double Bruhat graph is a finite graph associated with the finite Weyl group $W$, and it generalizes the aforementioned quantum Bruhat graph. We compare the double Bruhat graph with some foundational literature on the quantum Bruhat graph in Section~\ref{sec:DBG}.

Thus, our first main result expresses the intersections of semi-infinite orbits using the double Bruhat graph.
\begin{theorem}[{Cf.\ Theorem~\ref{thm:semiInfiniteOrbitsViaPaths}}]\label{thm:introSIO}
Let $x,y\in\widetilde W$ and $u,v\in W$. Denote by $w_0\in W$ the longest element. Then the intersection
\begin{align*}
\prescript u{}{}U(\breve F)xI\cap (\prescript {uw_0}{}{}U(\breve F)\cap \prescript v{}{}U(\breve F))yI~\subset~ G(\breve F)/I
\end{align*}
has finite dimension (or is empty). We provide a decomposition into finitely many locally closed subsets of the affine flag variety, parametrized by certain paths in the double Bruhat graph. Each subset is irreducible, smooth and we calculate its dimension explicitly.
\end{theorem}

Finally, in Section~\ref{sec:ADLV}, we apply these results on semi-infinite orbits to questions on affine Deligne--Lusztig varieties. We review the theory of \cite{Goertz2006} and study the approximation of $IxI$ by semi-infinite orbits. We introduce a new regularity condition on elements $x\in \widetilde W$ that we call \emph{superparabolic}. While this is a fairly restricting assumption, it covers in a quantitative sense \enquote{most} elements in the extended affine Weyl group.

\begin{theorem}[{Cf.\ Theorem~\ref{thm:adlvViaSemiInfiniteOrbits}}]\label{thm:introAdlv}
Let $x\in\widetilde W$ and choose an integral element $[b]\in B(G)$. We give a necessary condition for $X_x(b)\neq\emptyset$ and an upper bound $d$ for its dimension, both in terms of the double Bruhat graph. This improves previously known estimates such as Mazur's inequality or He's virtual dimension. We also give an upper bound for the number of $J_b(F)$-orbits of $d$-dimensional irreducible components of $X_x(b)$.

If $x=wt^\mu$ is superparabolic, and $\mu^{\dom} - \nu(b)$ is small relative to the superparabolicity condition imposed, then the above \enquote{necessary condition} for $X_x(b)\neq\emptyset$ becomes sufficient, and the above upper bound for the dimension is sharp, i.e.\ $\dim X_x(b) = d$. If moreover the Newton point of $[b]$ is regular, then the above upper bound for the number of irreducible components is sharp.
\end{theorem}

If $x$ is in a shrunken Weyl chamber, the superparabolicity condition is simply a superregularity condition like the ones typically studied in the literature, e.g.\ \cite{Milicevic2021, Milicevic2020, He2021d}. While some affine Deligne--Lusztig varieties associated with superregular elements $x$ have been described in the past, this was only possible in the cases where $[b]\in B(G)$ is either the largest element $[b_x]\in B(G)_x$ or relatively small in $B(G)_x$ (in the sense of Theorem \ref{thm:virtDim} (b)). Our result \enquote{fills the gap}, describing the geometry of $X_x(b)$ when $[b]$ is relatively large with respect to $x$. 

Moreover, there are plenty of superparabolic elements which do not lie in any shrunken Weyl chamber. In fact, in a quantitative sense, \enquote{most} elements which do not lie in shrunken Weyl chambers are superparabolic. These cases have rarely been studied in the past, and the geometry of $X_x(b)$ has only been understood in very specific situations (such as $x$ being cordial or $[b] = [b_x]$). Theorem~\ref{thm:introAdlv} fully answers our main questions for superparabolic $x\in\widetilde W$ and many $[b]\in B(G)$.

Theorem~\ref{thm:introAdlv} crucially assumes that $F$ is of equal characteristic, the group $G$ is split and the element $[b]$ is integral. The assumption on $F$ can easily be removed from the theorem by using formal arguments comparing the equal characteristic case with the mixed characteristic case, cf.\ \cite[Section~6.1]{He2014}. It is reasonable to expect that the assumption of $G$ being split can be lifted if one finds an appropriate generalization of \cite{Goertz2006} to non-split groups. It is unfortunately unclear how to lift the assumption of $[b]$ being integral for the method of this paper to work. The generalization of \cite{Goertz2006} to non-integral $\sigma$-conjugacy classes $[b]$ is given in \cite{Goertz2010}, but the connection between the latter paper and the double Bruhat graph remains unclear. In the second part of this two paper series, we will consider a different approach towards the geometry of $X_x(b)$. This approach comes without any assumptions on $G, F, [b]$, but requires superregular elements $x\in\widetilde W$ instead of the more permissible notion of superparabolic elements considered here.

By introducing the double Bruhat graph, we can capture the delicate interplay between $x\in \widetilde W$ and $[b]\in B(G)$, which is not accounted for e.g.\ by the notion of virtual dimension. Using this new language, we give new insights on the geometry of affine Deligne--Lusztig varieties, filling a conceptual vacuum of what $\dim X_x(b)$ \enquote{should be} when it cannot be virtual dimension. In this paper and its sequel, we hope to give a glimpse on how a generalization of Theorem~\ref{thm:hyperspecial} to the Iwahori level might look like. 

\section{Acknowledgements}
The author was partially supported by the German Academic Scholarship Foundation, the Marianne-Plehn programme, the DFG Collaborative Research Centre 326 \emph{GAUS} and the Chinese University of Hong Kong. I would like to thank Eva Viehmann, Xuhua He and Quingchao Yu for inspiring discussions, and Eva Viehmann again for her comments on a preliminary version of this article. I am very grateful for the carefuly reading and helpful comments by the annonymous referee.
\section{Notation}
Let $\mathbb F_q$ be a finite field and $F = \mathbb F_q\doubleparen t$ the field of formal Laurent series. We denote the usual $t$-adic valuation by $\nu_t$. Then $\mathcal O_F = \mathbb F_q\doublebrack t$ is its ring of integers. Choose an algebraic closure $k = \overline{\mathbb F_q}$ and denote by $L = \breve F = k\doubleparen t$ the completion of the maximal unramified extension of $F$. We write $\mathcal O_L = k\doublebrack t$ for its ring of integers. Denote the Frobenius of $L/F$ by $\sigma$, i.e.\
\begin{align*}
\sigma\Bigl( \sum\nolimits_i a_i t^i\Bigr) = \sum\nolimits_i (a_i)^{q}\, t^i.
\end{align*}

We consider a split reductive group $G$ defined over $\mathcal O_F$. We fix a split maximal torus and a Borel $T\subset B\subset G$ both defined over $\mathcal O_F$. As our Iwahori subgroup $I$, we choose the preimage of $B(k)$ under the projection $G(\mathcal O_L)\rightarrow G(k)$.

Denote the (co)character lattices of $T$ by $X^\ast(T)$ resp.\ $X_\ast(T)$, and the (co)root systems by $\Phi\subset X^\ast(T),\Phi^\vee\subset X_\ast(T)$. The positive roots defined by $B$ are denoted $\Phi^+$. We let $W = N_G(T)/T$ be the Weyl group of $W$ and $\widetilde W = N_G(T)(L)/T(\mathcal O_{L})$ the extended affine Weyl group. Under the isomorphism $\widetilde W\cong W\ltimes X_\ast(T)$, we write elements $x\in \widetilde W$ as $x=w t^\mu$ for $w\in W, \mu\in X_\ast(T)$.

Denote by $U$ the unipotent radical of $B$, so that $B = UT$. For each $\alpha\in \Phi$, we denote the corresponding root subgroup by $U_\alpha\subset G$. These come with an isomorphism to $U_\alpha\cong G_a$, the one-dimensional additive group over $F$, from the construction of the Bruhat-Tits building.

The set of affine roots is $\Phi_\af = \Phi\times\mathbb Z$. For each affine root $a=(\alpha,n)$, we define the affine root subgroup $U_a\subset U_\alpha(L)$ to be the set of elements of the form $U_\alpha(rt^n)$ with $r\in k$. The natural action of $\widetilde W$ on $\Phi_\af$ is given by
\begin{align*}
(w t^\mu)(\alpha,n) = (w\alpha,n-\langle\mu,\alpha\rangle).
\end{align*}

We denote the positive affine roots by $\Phi_\af^+$, these are those $a\in \Phi_\af$ with $U_a\subset I$. By abuse of notation, we denote the indicator function of positive roots by $\Phi^+$ as well. Then
\begin{align*}
a = (\alpha,n)\in \Phi_\af^+\iff n\geq\Phi^+(-\alpha) := \begin{cases}1,&\alpha\in \Phi^-,\\0,&\alpha\in \Phi^+.\end{cases}
\end{align*}
Denote the set of simple roots by $\Delta\subseteq \Phi^+$ and the set of simple affine roots by $\Delta_\af\subseteq \Phi_\af^+$. The latter are given by the roots of the form $(\alpha,0)$ for $\alpha\in \Delta$ as well as $(-\theta,1)$ whenever $\theta$ is the highest root of an irreducible component of $\Delta$.

For $x=w t^\mu\in \widetilde W$, we denote by $\LP(x)\subseteq W$ the set of length positive elements as introduced by \cite[Section~2.2]{Schremmer2022_newton}. We remark that $\LP(x)$ is always non-empty, and it collapses to one single element if and only if $x$ satisfies a mild regularity condition known as a \emph{shrunken Weyl chamber} \cite[Definition 7.2.1]{Goertz2010}. This is equivalent to $x$ lying in the lowest two--sided Kazhdan--Lusztig cell. If $\LP(x) = \{v\}$, then the element $\eta(x)$ occurring in the definition of virtual dimension above is given by $v^{-1} wv$.
\section{Semi-infinite orbits}\label{sec:semiInfiniteOrbits}
For any $u\in W$, the affine flag variety can be decomposed into semi-infinite orbits
\begin{align*}
G(L)/I = \bigsqcup_{x\in \widetilde W} \prescript u{}{}U(L)xI.
\end{align*}
Each element of the finite Weyl group $W$ yields a different decomposition of $G(L)/I$, so one may naturally ask how these decompositions are related. Given $u,v\in W$ and $x,y\in\widetilde W$, we would like to understand
\begin{align*}
\prescript u{}{}U(L)xI~\cap~\prescript v{}{}U(L)yI~\subset~G(L)/I.
\end{align*}
Up to multiplying both sides by $x^{-1}$ on the left and re-labelling, it suffices to study intersections
\begin{align*}
\prescript {uw_0}{}U(L)yI~\cap~\prescript v{}{}U(L)I~\subset~G(L)/I.
\end{align*}
Here, we write $w_0\in W$ for the longest element of the Weyl group, such that $\prescript{uw_0}{}{}B(L)$ is the Borel subgroup opposite to $\prescript u{}{}B(L)$.

Let us enumerate the positive roots as $\Phi^+ = \{\beta_1,\dotsc,\beta_{\#\Phi^+}\}$.
Then every element $g\in \prescript v{}{}U(L)$ can be written in the form $g = U_{v\beta_1}(g_1)\cdots U_{v\beta_{\#\Phi^+}}(g_{\#\Phi^+})$ with $g_1,\dotsc,g_{\#\Phi^+}\in L$. For each such element $g$, there exists a uniquely determined $y\in\widetilde W$ with $gI \in\prescript {uw_0}{}{}U(L)yI$, and we wish to compute that element $y$ in terms of the $g_i\in L$.

In order to facilitate this computation, we make two simplifications.
First, let us restrict the enumeration of positive roots $\Phi^+ = \{\beta_1,\dotsc,\beta_{\#\Phi^+}\}$ such that
\begin{align*}
\{\beta\in \Phi^+\mid (uw_0)^{-1}v\beta\in \Phi^+\} = \{\beta_{1},\dotsc,\beta_{n}\}
\end{align*}
for some $n\in\{0,\dotsc,\#\Phi^+\}$. Then
\begin{align*}
U_{v\beta_1}(g_1)\cdots U_{v\beta_n}(g_n)\in \prescript {uw_0}{}{}U(L)
\end{align*}
by choice of the labelling of the positive roots. Hence we may replace $g$ by \begin{align*}g' = U_{v\beta_{n+1}}(g_{n+1})\cdots U_{v\beta_{\#\Phi^+}}(g_{\#\Phi^+}),\end{align*}
using that $gI\in \prescript {uw_0}{}{}U(L)yI$ if and only if $g'I \in \prescript {uw_0}{}{}U(L)yI$.

For now, we expressed $g'\in \prescript v{}{}U(L)\cap \prescript{u}{}{}U(L)$ using an arbitrary enumeration of the roots
\begin{align*}
\{\beta\in \Phi^+\mid u^{-1} v\beta\in \Phi^+\} = \{\beta_{n+1},\dotsc,\beta_{\#\Phi^+}\}.
\end{align*}
Our second simplification is to use not just any such enumeration, 
but rather a specific one with extra structure, namely a \emph{reflection order}.
\begin{lemma}[{\cite[Proposition~2.13]{Dyer1993}, \cite{Papi1994}}]\label{lem:dyer}Let $\prec$ be a total order on $\Phi^+$. Then the following are equivalent:
\begin{enumerate}[(a)]
\item For all $\alpha,\beta\in \Phi^+$ with $\alpha+\beta\in\Phi^+$, we have
\begin{align*}
\alpha\prec\alpha+\beta\prec\beta\text{ or }\beta\prec\alpha+\beta\prec\alpha.
\end{align*}
\item There exists a uniquely determined reduced word for the longest element $w_0 = s_{\alpha_1}\cdots s_{\alpha_{\#\Phi^+}}$ with corresponding simple roots $\alpha_1,\dotsc,\alpha_{\#\Phi^+}\in\Delta$ such that
\begin{align*}
&\alpha_1\prec s_{\alpha_1}(\alpha_2)\prec\cdots\prec s_{\alpha_1}\cdots s_{\alpha_{\#\Phi^+-1}}(\alpha_{\#\Phi^+}).\rightqed
\end{align*}
\end{enumerate}
\end{lemma}
A total order satisfying these equivalent conditions is called a \emph{reflection order}. The following important facts on reflection orders will be used frequently.
\begin{lemma}\label{lem:reflectionOrderProperties}
Let $\prec$ be a reflection order and write $\Phi^+ =\{\beta_1\prec\cdots\prec\beta_{\#\Phi^+}\}$. 
\begin{enumerate}[(a)]
\item For $1\leq a\leq b\leq \#\Phi^+$ and $u\in W$, the subsets
\begin{align*}
&U_{u\beta_a}(L) U_{u\beta_{a+1}}(L)\cdots U_{u\beta_b}(L)~\text{and}
\\&U_{u\beta_{b+1}}(L) \cdots U_{u\beta_{\#\Phi^+}}(L) U_{-u\beta_1}(L) U_{-u\beta_2}(L)\cdots U_{-u \beta_{a-1}}(L)
\end{align*}
are subgroups of $G(L)$.
\item For $n=0,\dotsc,\#\Phi^+$ and $u = s_{\beta_{n+1}}\cdots s_{\beta_{\#\Phi^+}}\in W$, we have
\begin{align*}
\{\beta\in \Phi^+\mid u^{-1}\beta\in \Phi^-\} = \{\beta_{n+1},\dotsc,\beta_{\#\Phi^+}\}.
\end{align*}
Any $u\in W$ arises in this way for some reflection order $\prec$ and some index $n\in\{0,\dotsc,\#\Phi^+\}$.
\end{enumerate}
\end{lemma}
\begin{proof}
\begin{enumerate}[(a)]
\item If $\alpha,\beta\in \{\beta_{a},\dotsc,\beta_b\}$, then any positive linear combination of $\alpha,\beta$ that lies in $\Phi^+$ will also lie in this set. The fact that the first subset of $G(L)$ is a subgroup thus follows from the known theory of root subgroups \cite[Proposition~8.2.3]{Springer1998}.

Let us study the second subset. If both $\alpha,\beta$ lie in $\{\beta_{b+1},\dotsc,\beta_{\#\Phi^+}\}$, or both lie in $\{-\beta_1,\dotsc,-\beta_{a-1}\}$, so will their sum (if it is in $\Phi$). So suppose that $\alpha\in\{\beta_{b+1},\dotsc,\beta_{\#\Phi^+}\}$ and $\beta\in\{-\beta_1,\dotsc,-\beta_{a-1}\}$ satisfy $\alpha+\beta\in\Phi$.

If $\alpha+\beta\in\Phi^+$, then $\alpha = (\alpha+\beta)+(-\beta)$ is expressed as the sum of two positive roots, which cannot both be $\prec\alpha$. Hence $\alpha+\beta\succ\alpha$, thus $\alpha+\beta\in\{\beta_{b+1},\dotsc,\beta_{\#\Phi^+}\}$ as well.

If $\alpha+\beta\in\Phi^-$, then $-\beta = \alpha-(\alpha+\beta)$ is expressed as the sum of two positive roots, which cannot both be $\succ-\beta$. Hence $-(\alpha+\beta)\prec-\beta$, so $\alpha+\beta\in\{-\beta_1,\dotsc,-\beta_{a-1}\}$. The claim follows as above.
\item Let $w_0 = s_{\alpha_1}\cdots s_{\alpha_{\#\Phi^+}}$ be the reduced word such that $\beta_i = s_{\alpha_1}\cdots s_{\alpha_{i-1}}(\alpha_i)$ for $i=1,\dotsc,\#\Phi^+$. Then
\begin{align*}
uw_0 &= s_{\beta_{n+1}}\cdots s_{\beta_{\#\Phi^+}} \cdots s_{\beta_{\#\Phi^+}}\cdots s_{\beta_1}
\\&=s_{\beta_n}\cdots s_{\beta_1} = s_{\alpha_1}\cdots s_{\alpha_n}.
\end{align*}
Hence
\begin{align*}
\{\beta\in \Phi^+\mid u^{-1}\beta\in \Phi^-\} &= \{\beta\in \Phi^+\mid (uw_0)^{-1}\beta\in \Phi^+\}
\\&=\Phi^+\setminus \{\beta\in \Phi^+\mid (uw_0)^{-1}\beta\in \Phi^-\}
\\&=\Phi^+\setminus\{\alpha_1,s_{\alpha_1}(\alpha_2),\dotsc, s_{\alpha_1}\cdots s_{\alpha_{n-1}}(\alpha_n)\}
\\&=\Phi^+\setminus\{\beta_1,\dotsc,\beta_n\} = \{\beta_{n+1},\dotsc,\beta_{\#\Phi^+}\}.
\end{align*}
This shows the first claim. Now for any given $u\in W$, we can find some reduced word $uw_0 = s_{\alpha_1}\cdots s_{\alpha_n}$. Continue it to the right to a reduced word for $w_0$ to obtain the desired reflection order.\qedhere
\end{enumerate}
\end{proof}
So when studying intersections as above, i.e.\
\begin{align}
\prescript {uw_0}{}{}U(L)yI~\cap~(\prescript u{}{}U(L)\cap \prescript v{}{}U(L))I~\subset~G(L)/I,\label{eq:typicalIntersectionSIO}
\end{align}
we may write
\begin{align*}
\prescript{u}{}U(L)\cap \prescript v{}U(L) = U_{u\beta_1}(L)\cdots U_{u\beta_n}(L)
\end{align*}
for a suitable reflection order $\beta_1\prec\cdots\prec\beta_{\#\Phi^+}$. With this notation, the fundamental method to evaluate intersections as in \eqref{eq:typicalIntersectionSIO} is given by the following lemma.

\begin{lemma}\label{lem:semiInfiniteConjugation}
Let $\prec$ be a reflection order and write $\Phi^+ = \{\beta_1\prec\cdots\prec\beta_{\#\Phi^+}\}$. Let $x\in\widetilde W,u\in W$ and $1\leq n\leq \#\Phi^+$.
Consider an element of the form
\begin{align*}
g = U_{u\beta_1}(g_1)\cdots U_{u\beta_{n}}(g_{n})\in \prescript u{}{}U(L).
\end{align*}
Denote $m = \nu_L(g_{n})\in\mathbb Z$, $b = (u\beta_{n},m)\in \Phi_\af$ and $u' = us_{\beta_{n}}$.
\begin{enumerate}[(a)]
\item If $x^{-1}b\in \Phi_\af^+$, then
\begin{align*}
gxI = 
U_{u\beta_1}(g_1)\cdots U_{u\beta_{n-1}}(g_{n-1})xI/I\in G(L)/I.
\end{align*}
\item If $x^{-1}b\in \Phi_\af^-$, then there are polynomials $f_{1},\dotsc,f_{n-1}$ with
\begin{align*}
f_i \in \mathbb Z[X_{i+1},\dotsc,X_{n-1},Y],
\end{align*}
allowing us to write
\begin{align*}
&gxI \in U_{-u\beta_n}(L)\cdots U_{-u\beta_{\#\Phi^+}}(L)U_{u\beta_1}(\tilde g_1)\cdots U_{u\beta_{n-1}}(\tilde g_{n-1}) r_b xI/I \subset G(L)/I,
\end{align*}
where
\begin{align*}
\tilde g_i = g_i + f_i(g_{i+1},\dotsc,g_{n-1},g_{n}^{-1})\in L.
\end{align*}
The polynomial $f_i$ is a sum of monomials
\begin{align*}
\varphi X_{i+1}^{e_{i+1}}\cdots X_{n-1}^{e_{n-1}}Y^{f}
\end{align*}
satisfying the conditions $\varphi\in\mathbb Z$ and
\begin{align*}
e_{i+1}\beta_{i+1}+\cdots +e_{n-1}\beta_{n-1} - f\beta_{n} = \beta_i.
\end{align*}
It depends only on the datum of $G,B,T,\prec$, but not on $g$ nor $x$. We have
\begin{align*}
&U_{u\beta_1}(g_1)\cdots U_{u\beta_{n-1}}(g_{n-1}) U_{-u\beta_n}(g_n^{-1})\\&\qquad \in U_{-u\beta_n}(g_n^{-1}) U_{-u\beta_{n+1}}(L)\cdots U_{-u\beta_{\#\Phi^+}}(L) U_{u\beta_1}(\tilde g_1)\cdots U_{u\beta_{n-1}}(\tilde g_{n-1}).
\end{align*}
\end{enumerate}
\end{lemma}
\begin{proof}
The statement in (a) is immediately verified, since $x^{-1}(u\beta_n,\nu_L(g_n))\in \Phi_\af^+$ is equivalent to $\prescript {x^{-1}}{}{}U_{u\beta_n}(g_n)\in I$. So let us prove (b).

Using the fact $x^{-1}(-b)\in \Phi_\af^+$, we get
\begin{align*}
U_{u\beta_{n}}(g_{n})xI &= U_{u\beta_{n}}(g_{n}) U_{-u\beta_{n}}(-g_{n}^{-1})xI.
\intertext{Following the usual combinatorics of root subgroups, e.g.\ \cite[Lemma~8.1.4]{Springer1998} or \cite[Equation~(7.6)]{Parkinson2009}, we re-write this as}
\cdots&=U_{-u\beta_{n}}(g_{n}^{-1})U_{-u\beta_{n}}(-g_{n}^{-1})U_{u\beta_{n}}(g_{n}) U_{-u\beta_{n}}(-g_{n}^{-1})xI
\\&=U_{-u\beta_{n}}(g_{n}^{-1}) (-u\beta_{n})^\vee(-g_{n}^{-1}) n_{-u\beta_{n}}xI
\\&=U_{-u\beta_{n}}(g_{n}^{-1}) r_b xI.
\end{align*}
Here, the cocharacter $(-u\beta_{n})^\vee$ is understood as function $L\rightarrow T(L)$ and $n_{-u\beta_{n}} \in N_G(T)(L)$ is a representative of the reflection $s_{-u\beta_{n}}\in W$.

It remains to evaluate
\begin{align*}
g' := U_{u\beta_1}(g_1)\cdots U_{u\beta_{n-1}}(g_{n-1}) U_{-u\beta_{n}}(g_{n}^{-1}) \in \prescript{u'}{}{}U(L),
\end{align*}
where we write $u' = u s_{\beta_n}\cdots s_{\beta_{\#\Phi^+}}\in W$.
By \cite[Proposition~8.2.3]{Springer1998}, we may write
\begin{align*}
&U_{u\beta_{n-1}}(g_{n-1}) U_{-u\beta_{n}}(g_{n}^{-1}) =  U_{-u\beta_{n}}(g_{n}^{-1})\Bigl[\prod_{i,j} U_{\beta_{i,j}}(c_{i,j} g_{n}^{-i} g_{n-1}^j)\Bigr]
U_{u\beta_{n-1}}(g_{n-1}),
\end{align*}
where the product is taken over all indices $i,j\in\mathbb Z_{\geq 1}$ with $\beta_{i,j}:=-iu\beta_{n} + ju\beta_{n-1}\in\Phi$. The product can be evaluated in any fixed order, up to changing the structure constants $c_{i,j}$.
By the construction of the Bruhat-Tits building, the structure constants are in $\mathbb Z$, cf.\ \cite[Example 6.1.3 (b)]{Bruhat1972} or \cite[Chapter~9]{Springer1998}.

We want to iterate this procedure. We claim for all $1\leq j\leq n$ that we can write
\begin{align*}
&U_{u\beta_j}(g_j)\cdots U_{u\beta_{n-1}}(g_{n-1}) U_{-u\beta_n}(g_n^{-1}) \\&= U_{-u\beta_n}(g_n^{(j)})\cdots U_{-u\beta_{\#\Phi^+}}(g_{\#\Phi^+}^{(j)}) U_{u\beta_1}(g_1^{(j)})\cdots U_{u \beta_{n-1}}(g_{n-1}^{(j)})\tag{$\ast$}
\end{align*}
subject to the conditions
\begin{align*}
g_n^{(j)} &= g_n^{-1},
\\g_i^{(j)}&=g_i + f_i^{(j)}(g_{i+1},\dotsc,g_{n-1}, g_n^{-1})\text{ for }j\leq i<n,
\\g_i^{(j)}&=f_i^{(j)}(g_{j+1},\dotsc,g_{n-1}, g_n^{-1})\text{ for }1\leq i< j\text{ or }n<i\leq \#\Phi^+.
\end{align*}
Here, the polynomials $f_i^{(j)}$ are required to have the analogous properties as claimed in the lemma, i.e.\ the monomial $\varphi X_{j+1}^{e_{j+1}}\cdots X_{n-1}^{e_{n-1}} Y^f$ may only occur $f_i^{(j)}$ if
\begin{align*}
e_{j+1} \beta_{j+1}+\cdots + e_{n-1} \beta_{n-1} - f \beta_n = \begin{cases}\beta_i,& i<n,\\
-\beta_i,& i>n.
\end{cases}
\end{align*}

This long claim is trivially verified for $j=n$. In an inductive step, assume it has been proved for some $1<j\leq n$. We multiply the right--hand side of $(\ast)$ by $U_{u\beta_{j-1}}(g_{j-1})$ and apply \cite[Proposition~8.2.3]{Springer1998} to sort the resulting product into our usual order. By Lemma~\ref{lem:reflectionOrderProperties}, the result indeed lies in $U_{-u\beta_n}(L)\cdots U_{-u\beta_{\#\Phi^+}}(L)U_{u\beta_1}(L)\cdots U_{u\beta_{n-1}}(L)$, so this defines the elements $g_i^{(j-1)}\in L$ for $i=1,\dotsc,\#\Phi^+$.

It is straightforward to see (but cumbersome to write down in full details) that our required conditions for the $g^{(j-1)}_i$ are true precisely because they are true for the $g^{(j)}_i$. This finishes the induction. Specializing to $j=1$ proves the lemma.
\end{proof}
We want to iterate this lemma. Doing so, we obtain the following result.
\begin{proposition}\label{prop:semiInfiniteTypes}
Let $\Phi^+ = \{\beta_1\prec\cdots\prec\beta_{\#\Phi^+}\}$ and $u\in W$. Pick $gI \in\prescript{u}{}U(L)I/I$.
For each $n=0,\dotsc,\#\Phi^+$, consider the Borel subgroup of $G$ associated with the element $us_{\beta_{n+1}}\cdots s_{\beta_{\#\Phi^+}}\in W$ and the corresponding decomposition of the affine flag variety into semi-infinite orbits.
This allows us to define
 $x_0,\dotsc,x_{\#\Phi^+}=1\in\widetilde W$ to be the uniquely determined elements such that
\begin{align*}
gI \in \Bigl(\prescript{u s_{\beta_{n+1}}\cdots s_{\beta_{\#\Phi^+}}}{}{}U(L)\Bigr)x_{n} I/I,\qquad n=0,\dotsc,\#\Phi^+.
\end{align*}
Define
\begin{align*}
\{n_1<\cdots<n_N\} :=\{n\in\{1,\dotsc,\#\Phi^+\}\mid x_n\neq x_{n-1}\}.
\end{align*}
Choose a representative of $gI$ in $\prescript u{}{}U(L)$ and write
\begin{align*}
gI = U_{u\beta_1}(g_1)\cdots U_{u\beta_{\#\Phi^+}}(g_{\#\Phi^+}) I,\qquad g_1,\dotsc,g_{\#\Phi^+}\in L.
\end{align*}
Then for $n=0,\dotsc,\#\Phi^+$, we have the following:
\begin{enumerate}[(a)]
\item We may write
\begin{align*}
gI \in U_{-u\beta_{n+1}}(L)\cdots U_{-u\beta_{\#\Phi^+}}(L) U_{u\beta_1}(g^{(n)}_1)\cdots U_{u\beta_{n}}(g^{(n)}_{n})x_n I
\end{align*}
for elements $g^{(n)}_1,\dotsc,g^{(n)}_{n}\in L$, which are determined uniquely through polynomial identities
\begin{align*}
g^{(n)}_i - g_i = f^{(n)}_i(g_{i+1}^{(i+1)},\dotsc,g_{\#\Phi^+}^{(\#\Phi^+)}),\qquad f^{(n)}_i \in \mathbb Z[X_{i+1}^{\pm 1},\dotsc,X_{\#\Phi^+}^{\pm 1}]
\end{align*}
subject to the following condition: The polynomial $f^{(n)}_i$ depends only on the datum of $G,B,T,u,\prec$ and the indices in $\{n_1,\dotsc,n_N\}\cap \{n+1,\dotsc,\#\Phi^+\}$.
 It is a sum of monomials
\begin{align*}
\varphi X_{i+1}^{e_{i+1}}\cdots X_{\#\Phi^+}^{e_{\#\Phi^+}},\qquad \varphi,e_{i+1},\dotsc,e_{\#\Phi^+}\in\mathbb Z
\end{align*}
subject to the conditions $\beta_i = e_{i+1}\beta_{i+1}+\cdots + e_{\#\Phi^+}\beta_{\#\Phi^+}$ and
\begin{align*}
\forall h\in\{i+1,\dotsc,\#\Phi^+\}:~e_h<0\implies h\in\{n+1,\dotsc,\#\Phi^+\}\cap \{n_1\dotsc,n_N\}.
\end{align*}
\item Suppose that $n\geq 1$. If $g^{(n)}_n=0$, then $n\notin\{n_1,\dotsc,n_N\}$ and $x_{n-1} = x_n$
Otherwise, define \begin{align*}b_n := (u\beta_{n},\nu_L(g^{(n)}_{n}))\in\Phi_\af.\end{align*}

If $x_n^{-1}(b_n)\in\Phi_\af^+$, then $n\notin\{n_1,\dotsc,n_N\}$ and $x_{n-1} = x_n$.

If $x_n^{-1}(b_n)\in\Phi_\af^-$, then $n\in\{n_1,\dotsc,n_N\}$ and $x_{n-1} = r_b x_n$.
\item The values of $\nu_L(g^{(n_h)}_{n_h})\in\mathbb Z$ for $h\in\{1,\dotsc,N\}$ depend only on $gI\in G(L)/I$, and not on the chosen representative $U_{u\beta_1}(g_1)\cdots U_{u\beta_{\#\Phi^+}}(g_{\#\Phi^+})\in G(L)$.
\end{enumerate}
\end{proposition}
\begin{proof}
By Lemma~\ref{lem:reflectionOrderProperties}, we get
\begin{align*}
s_{\beta_{n+1}}\cdots s_{\beta_{\#\Phi^+}}\Phi^+ =\{-\beta_{n+1},\dotsc,-\beta_{\#\Phi^+},\beta_1,\dotsc,\beta_{n}\}.
\end{align*}
So indeed
\begin{align*}
\prescript{u s_{\beta_{n+1}}\cdots s_{\beta_{\#\Phi^+}}}{}{}U(L) = U_{-u\beta_{n+1}}(L)\cdots U_{-u\beta_{\#\Phi^+}}(L) U_{u\beta_1}(L)\cdots U_{u\beta_{n}}(L),
\end{align*}
as claimed indirectly in (a).

We explain how to find the elements $g^{(n)}_i$ via induction on $\#\Phi^+-n$, proving (b) along the way. For the inductive start, note that we have to choose $f^{(\#\Phi^+)}_\bullet\equiv 0$, so that $g^{(\#\Phi^+)}_i = g_i$.

For the inductive step, suppose now that we have constructed the elements $g^{(n)}_1,\dotsc,g^{(n)}_n$ for some $n\in\{1,\dotsc,\#\Phi^+\}$. Define $b_n$ as in (b).
If $(x_n)^{-1}b_n\in\Phi_\af^+$, we may apply Lemma \ref{lem:semiInfiniteConjugation} (a) to 
\begin{align*}
gI \in U_{-u\beta_{n+1}}(L)\cdots U_{-u\beta_{\#\Phi^+}}(L) U_{u\beta_1}(g^{(n)}_1)\cdots U_{u\beta_{n}}(g^{(n)}_{n})x_n I.
\end{align*}
By choice of $x_{n-1}$, we get $x_{n-1} = x_n$. We set $g^{(n-1)}_i = g^{(n)}_i$ for $i=1,\dotsc,n-1$.

If $(x_n)^{-1}b_n\in \Phi_\af^-$, we may apply Lemma \ref{lem:semiInfiniteConjugation} (b) to see
\begin{align*}
gI \in U_{-u\beta_{n}}(L)\cdots U_{-u\beta_{\#\Phi^+}}(L) U_{u\beta_1}({\tilde g}^{(n)}_1)\cdots U_{u\beta_{n-1}}({\tilde g}^{(n)}_{n-1})r_b x_nI.
\end{align*}
By choice of $x_{n-1}$, we get $x_{n-1} = r_b x_n$. In particular $n\in\{n_1,\dotsc,n_N\}$. The elements ${\tilde g}^{(n)}_\bullet$ are polynomials in the $g^{(n)}_\bullet$ as in Lemma \ref{lem:semiInfiniteConjugation} (b). We set $g^{(n-1)}_i := {\tilde g}^{(n)}_i$.

Observe that in any case, the value of $g^{(n-1)}_i - g^{(n)}_i$ is a polynomial with integer coefficients in $g^{(n)}_{i+1},\dotsc,g^{(n)}_{n-1}, (g^{(n)}_{n-1})^{-1}$ subject to the conditions of Lemma \ref{lem:semiInfiniteConjugation} (b). Using a simple induction on $\#\Phi^+ - i$, one can now see that $g^{(n)}_i - g_i$ has the desired shape as claimed in (a), by composition of these polynomials.

If $n\in\{n_1,\dotsc,n_N\}$, the value of $b_n$ is uniquely determined by $x_{n} x_{n+1}^{-1}$, which in turn is determined by $gI\in G(L)/I$ alone. Hence (c) follows.
\end{proof}
\begin{corollary}\label{cor:semiInfiniteTrivialization}
In the setting of Proposition~\ref{prop:semiInfiniteTypes}, we have $gI = I$ if and only if $U_{u\beta_n}(g_n)\in I$ for $n=1,\dotsc,\#\Phi^+$.
\end{corollary}
\begin{proof}
If each $U_{u\beta_n}(g_n)$ lies in $I$, then so does their product, hence $gI = I$.

If conversely $gI = I$, then all $x_n\in\widetilde W$ must be equal to $1$. Now part (b) of Proposition~\ref{prop:semiInfiniteTypes} shows that each $g_n$ must be zero or satisfy $x_n^{-1}(u\beta_n,\nu_L(g_n))\in \Phi_\af^+$. Since $x_n=1$, the latter condition is equivalent to $U_{u\beta_n}(g_n)\in I$.
\end{proof}
\begin{definition}
Let $\prec,x_0,u,gI$ be as in Proposition~\ref{prop:semiInfiniteTypes}. We define the \emph{semi-infinite type} of $gI$ to be the set
\begin{align*}
\sitype(gI) = \{(n_h, \nu_L(g^{(n_h)}_{n_h}))\mid h=1,\dotsc,N\}\subset \mathbb Z\times\mathbb Z.
\end{align*}
Any subset of $\mathbb Z\times\mathbb Z$ of the above form is called an \emph{admissible type} for $(x_0,u,\prec)$.
\end{definition}

\begin{lemma}\label{lem:typeVarieties}
Let $\Phi^+ = \{\beta_1\prec\cdots\prec\beta_{\#\Phi^+}\}$ be a reflection order and $u\in W$. Choose an arbitrary subset $\{n_1<\cdots <n_N\}\subseteq \{1,\dotsc,\#\Phi^+\}$ and values $\nu_h\in\mathbb Z$ for $h=1,\dotsc,N$. Define $b_h := (u\beta_{n_h}, \nu_h)\in\Phi_\af$.
\begin{enumerate}[(a)]
\item
The set $\{(n_1,\nu_1),\dotsc,(n_N, \nu_N)\}$ defines an admissible type for $(r_{b_1}\cdots r_{b_N},u,\prec)$ if and only if
\begin{align*}
r_{b_N}\cdots r_{b_{h+1}}(b_h) \in\Phi_{\af}^-
\end{align*}
for $h=1,\dotsc,N$.
\item There is a locally closed and reduced $k$-sub-ind-scheme 
\begin{align*}
\mathcal T = \mathcal T_{u,\prec,(n_1,\nu_1),\dotsc,(n_N,\nu_N)}
\end{align*} of the affine flag variety whose $k$-valued points are given by precisely those elements
\begin{align*}
gI \in U_{u\beta_1}(L)\cdots U_{u\beta_{\#\Phi^+}}(L)I/I\subset G(L)/I
\end{align*}
which satisfy $\sitype(gI) = \{(n_1,\nu_1),\dotsc,(n_N,\nu_N)\}$.
\end{enumerate}
\end{lemma}
\begin{proof}
\begin{enumerate}[(a)]
\item
The given condition for the $b_h$ is certainly necessary by Proposition~\ref{prop:semiInfiniteTypes}. Conversely, if this condition is satisfied, one may iteratively choose values for $g_i$ in Proposition~\ref{prop:semiInfiniteTypes} to construct an element $gI\in G(L)/I$ of the desired type.
\item The definition of $\sitype(gI)$ in terms of $L$-valuations allows us to write $gI\in T(k)$ in terms of vanishing or non-vanishing of certain polynomials over $k$. Hence we get a well-defined reduced subscheme with these geometric points.\qedhere
\end{enumerate}
\end{proof}
We call $\mathcal T_{u,\prec,(n_1,\nu_1),\dotsc,(n_N,\nu_N)}$ a \emph{type variety}. These type varieties are analogues of the Gelfand-Goresky-MacPherson-Serganova strata in the affine Grassmannian, cf.\ \cite{Kamnitzer2010}. We need the following numerical datum to describe their dimensions. This can be seen as finite a replacement for the infinite dimension of $\prescript u{}{}U(L)xI$.
\begin{definition}
Let $x=w t^\mu\in \widetilde W, u\in W$ and $\alpha\in \Phi$.
\begin{enumerate}[(a)]
\item We define the \emph{length functional}, following \cite[Definition~2.5]{Schremmer2022_newton}, as
\begin{align*}
\ell(x,\alpha) = \langle \mu,\alpha\rangle + \Phi^+(\alpha)-\Phi^+(w\alpha).
\end{align*}
\item We define
\begin{align*}
\ell_u(x) := \sum_{\alpha\in \Phi^+} \ell(x^{-1},u\alpha)=\langle -u^{-1} w\mu,2\rho\rangle - \ell(u) + \ell(w^{-1} u).
\end{align*}
\end{enumerate}
\end{definition}

The claimed identity can easily be seen along the lines of \cite[Corollary~2.11]{Schremmer2022_newton}. This result moreover proves that $\ell_u(x)\leq \ell(x)$, with equality holding if and only if $u\in\LP(x^{-1})$. From \cite[Lemma~2.9]{Schremmer2022_newton}, we see
\begin{align*}
\ell_u(x) = \dim ((I\cap \prescript u{}{}U(L))xI/I) - \dim((I\cap \prescript{uw_0}{}{}U(L))xI/I).
\end{align*}
\begin{proposition}\label{prop:Tgeom}
Let $\tau = \{(n_1,\nu_1),\dotsc,(n_N,\nu_N)\}$ be an admissible type for $(x,u,\prec)$. Then $\mathcal T = \mathcal  T_{u,\prec,(n_1,\nu_1),\dotsc,(n_N,\nu_N)}$ is a finite-dimensional irreducible smooth affine scheme over $k$. We have
\begin{align*}
\dim \mathcal T = \frac 12\left(N-\ell_u(x)\right).
\end{align*}
\end{proposition}
\begin{proof}
First consider the case $N=0$. Then evidently $\mathcal T$ is just a point over $k$, given by $\mathcal T(k) = \{I\}\subset G(L)/I$.

Suppose now $N\geq 1$. We prove the claim via induction on $n_N$ (with the inductive start being the case $n_N$ undefined, i.e.\ $N=0$ above). For $m\in \mathbb Z$, we define the \emph{truncation map} $\tr_{\leq m} : L\rightarrow L$ as
\begin{align*}
\tr_{\leq m}\Bigl(\sum_{i\in\mathbb Z} a_i t^i\Bigr) = \sum_{i\leq m}a_i t^i.
\end{align*}
Denote its image by $L_{\leq m}$, which is easily equipped with the structure of an $k$-ind-scheme. We define the map of $k$-ind-schemes
\begin{align*}
f_1 :\mathcal  T\rightarrow L_{\leq -\Phi^+(u\beta_1)},\qquad U_{u\beta_1}(g_1)\cdots U_{u\beta_{\#\Phi^+}}(g_{\#\Phi^+})I\in \mathcal T(k)~~\mapsto~~ \tr_{\leq -\Phi^+(u\beta_1)}(g_1).
\end{align*}
In order to check that this is well-defined, suppose that
\begin{align*}
U_{u\beta_1}(g_1)\cdots U_{u\beta_{\#\Phi^+}}(g_{\#\Phi^+})I = 
U_{u\beta_1}(\tilde g_1)\cdots U_{u\beta_{\#\Phi^+}}(\tilde g_{\#\Phi^+})I\in \mathcal T(k)
\end{align*}
for some $g_\bullet,\tilde g_\bullet\in L$. Then
\begin{align*}
\Bigl[U_{u\beta_1}(g_1)\cdots U_{u\beta_{\#\Phi^+}}(g_{\#\Phi^+})\Bigr]^{-1} \Bigl[U_{u\beta_1}(\tilde g_1)\cdots U_{u\beta_{\#\Phi^+}}(\tilde g_{\#\Phi^+})\Bigr] \in \prescript u{}{}U(L)\cap I.
\end{align*}
By Corollary~\ref{cor:semiInfiniteTrivialization} and the reflection order property, we conclude $U_{u\beta_1}(\tilde g_1-g_1)\in I$. Hence $\tr_{\leq -\Phi^+(u\beta_1)}(g_1) = \tr_{\leq -\Phi^+(u\beta_1)}(\tilde g_1)$. This shows well-definedness of the map $f_1$.

Define the reflection order $\prec' = \prec^{\beta_1}$ as in \cite[Proposition~5.2.3]{Bjorner2005}, so
\begin{align*}
s_{\beta_1}(\beta_2)\prec' s_{\beta_1}(\beta_3)\prec'\cdots \prec' s_{\beta_1}(\beta_{\#\Phi^+})\prec' \beta_1.
\end{align*}
Define moreover the type
\begin{align*}
\tau' = \{(n_i-1,\nu_i)\mid i \in\{1,\dotsc,N\}\text{ and }n_i>1\}.
\end{align*}
Write $\mathcal T' =\mathcal  T_{us_{\beta_1},\prec',\tau'}$. Then the inductive assumption applies to $\mathcal T'$. For all elements $U_{u\beta_1}(g_1)\cdots U_{u\beta_{\#\Phi^+}}(g_{\#\Phi^+})I\in \mathcal T(k)$, one easily checks $U_{u\beta_2}(g_2)\cdots U_{u\beta_{\#\Phi^+}}(g_{\#\Phi^+})I\in \mathcal T'(k)$.

Observe that $U_{u\beta_1}(L)$ normalizes $U_{u\beta_2}(L)\cdots U_{u\beta_{\#\Phi^+}}(L)$. By definition of the variety $\mathcal T'$, we see that $(U_{u\beta_1}(L)\cap I) \mathcal T'(k) = \mathcal T'(k)$.
Hence we obtain a well-defined map of $k$-ind-schemes $f_2 : \mathcal T\times (U_{u\beta_1}(L)\cap I)\rightarrow \mathcal T'$ sending $gI\in \mathcal T(k)$ and $U_{u\beta_1}(h)\in I$ to
\begin{align*}
U_{u\beta_1}(h-f_1(gI))gI\in \mathcal T'(k).
\end{align*}
Define
\begin{align*}
x' = w't^{\mu'} = \begin{cases} x,&n_1>1,\\
r_{b_1} x,&n_1=1,\end{cases}
\end{align*}
such that $\tau'$ is admissible for $(x',us_{\beta_1},\prec')$. Thus $\mathcal T'(k)\subset \prescript{us_{\beta_1} w_0}{}{}U(L)x'I$, allowing us to write elements $g'I\in \mathcal T'(k)$ in the form
\begin{align*}
g'I = U_{-u\beta_2}(g'_1)\cdots U_{-u\beta_{\#\Phi^+}}(g'_{\#\Phi^+-1}) U_{u\beta_1}(g'_{\#\Phi^+})x'I.
\end{align*}
Here, we have
\begin{align*}
U_{u\beta_1}(g'_{\#\Phi^+})\in \prescript {x'}{}{}I\iff& (x')^{-1}(u\beta_1,\nu_L(g'_{\#\Phi^+}))\in \Phi_\af^+
\\\iff&\nu_L(g'_{\#\Phi^+}) \geq \langle -w'\mu',u\beta_1 \rangle + \Phi^+(-(w')^{-1}u\beta_1)=:m+1\in\mathbb Z.
\end{align*}
Thus we obtain a well-defined morphism of $k$-ind-schemes $\varphi_1:\mathcal T'\rightarrow L_{\leq m}$ sending $g'I\in \mathcal T'(k)$ as represented above to $\tr_{\leq m}(g'_{\#\Phi^+})$ (check well-definedness using Corollary~\ref{cor:semiInfiniteTrivialization} as above).

Let $S\subset L$ be the $k$-sub-ind-scheme defined by the following condition for $z\in L$:
\begin{align*}
z\in S(k) :\iff \begin{cases}
\nu_L(z)\geq m+1,&\text{ if }n_1>1,\\
\nu_L(z) = \nu_1,&\text{ if }n_1=1.
\end{cases}
\end{align*}
We would like to define the map of $k$-ind-schemes $\varphi_2: \mathcal T'\times S\rightarrow \mathcal T$ sending $g'I\in \mathcal T'(k)$ and $z\in S(k)$ to
\begin{align*}
\varphi_2(g'I, z) = U_{u\beta_1}(z-\varphi_1(g'I)) g'I.
\end{align*}
Let us check that $\varphi_2$ is well-defined, i.e.\ takes values in $\mathcal T$ as claimed. For $i=2,\dotsc,\#\Phi^+$, we have
\begin{align*}
\prescript{us_{\beta_i}\cdots s_{\beta_{\#\Phi^+}}}{}{}U(L) \varphi_2(g'I,z)& = 
\prescript{us_{\beta_i}\cdots s_{\beta_{\#\Phi^+}}}{}{}U(L) g'I 
\\&=\prescript{(us_{\beta_1})s_{s_{\beta_1}(\beta_i)}\cdots s_{s_{\beta_1}(\beta_{\#\Phi^+})}s_{\beta_1}}{}{}U(L) g'I.
\end{align*}
Moreover, computing
\begin{align*}
\varphi_2(g'I,z)\in U_{-u\beta_2}(L)\cdots U_{-u\beta_{\#\Phi^+}}(L)U_{u\beta_1}(z)x'I,
\end{align*}
we can apply Lemma~\ref{lem:semiInfiniteConjugation} to get $\varphi_2(g'I,z)\in \prescript{uw_0}{}{}U(L)xI$ by the condition $z\in S(k)$. Comparing the definitions of $\tau$ and $\tau'$, we get $\varphi_2(g'I,z)\in \mathcal T(k)$.

For a sufficiently large integer $M\gg 0$, one checks that we have an isomorphism of $k$-ind-schemes
\begin{align*}
\mathcal T \times (U_{u\beta_1}(L_{\leq M})\cap I)\rightarrow \mathcal T' \times (S\cap L_{\leq M}),
\end{align*}
sending $gI\in \mathcal T(k)$ and $U_{u\beta_1}(h)\in U_{u\beta_1}(L_{\leq M})\cap I$ to
$f_2(gI,h)\in \mathcal T'(k)$ and
\begin{align*}
-h+\varphi_1(f_2(gI,h)) +f_1(gI)\in S(k)\cap L_{\leq M}.
\end{align*}
Its inverse is the map sending $g'I\in \mathcal T'(k)$ and $z\in S(k)\cap L_{\leq M}$ to $\varphi_2(g'I,z)\in \mathcal T(k)$ and
\begin{align*}
U_{u\beta_1}(-z + f_1(\varphi_2(g'I,z)) + \varphi_1(g'I))\in U_{u\beta_1}(L_{\leq M})\cap I.
\end{align*}
By the inductive assumption, $\mathcal T'$ is a finite-dimensional irreducible smooth affine scheme over $k$. The same conditions hold true for $S\cap L_{\leq M}$ (which is either an affine space over $k$ or the product of a pointed affine line with an affine space). Hence the same conditions all hold true for $\mathcal T\times (U_{u\beta_1}(L_{\leq M})\cap I)$. It follows that they must also hold true for $\mathcal T$ itself. Moreover, we have
\begin{align*}
\dim\mathcal  T' - \dim \mathcal T &= \dim (U_{u\beta_1}(L_{\leq M})\cap I)-\dim (S\cap L_{\leq M}) \\&= \begin{cases}m+1 - \Phi^+(-u\beta_1),&\text{ if } n_1>1,\\
\nu_1 - \Phi^+(-u\beta_1),&\text{ if } n_1=1.\end{cases}
\end{align*}
By induction, we know $2\dim \mathcal T' =N'-\ell_{us_{\beta_1}}(x')$, with $N' = N$ if $n_1>1$ and $N' = N-1$ if $n_1=1$. We show $2\dim\mathcal  T = N-\ell_u(x)$, using a case distinction depending on whether $n_1>1$ or not.

First consider the case $n_1=1$. From \cite[Lemma~2.12]{Schremmer2022_newton} or direct calculation, we get
\begin{align*}
\ell_{us_{\beta_1}}(x') = \ell_{us_{\beta_1}}(r_{b_1}x) = \ell_{us_{\beta_1}}(r_{b_1})+\ell_u(x).
\end{align*}
We calculate
\begin{align*}
\ell_{us_{\beta_1}}(r_{b_1}) &= \ell_{us_{\beta_1}}(s_{u\beta_1}t^{\nu_1u\beta_1^\vee}) = \langle -\nu_1\beta_1^\vee,2\rho\rangle-\ell(us_{\beta_1})+\ell(u).
\intertext{
Since $\beta_1$ is simple, the above expression simplifies to}
\cdots&=-2\nu_1 -1+2\Phi^+(-u\beta_1).
\end{align*}
We conclude
\begin{align*}
\dim\mathcal  T &= \dim \mathcal T' - \nu_1+\Phi^+(-u\beta_1) = \frac 12\left(N'-\ell_{us_{\beta_1}}(x')\right) - \nu_1 + \Phi^+(-u\beta_1)
\\&=\frac 12\left(N-\ell_u(x) + 2\nu_1 -2\Phi^+(-u\beta_1)\right) - \nu_1 + \Phi^+(-u\beta_1)
=\frac 12\left(N-\ell_u(x)\right).
\end{align*}
Let us now consider the case $n_1>1$. Then we calculate
\begin{align*}
\ell_{us_{\beta_1}}(x) &= \sum_{\alpha\in \Phi^+}\ell(x^{-1},us_{\beta_1}\alpha)
=\sum_{\alpha\in s_{\beta_1}\Phi^+}\ell(x^{-1},u\alpha) \\&= \ell_u(x) - \ell(x^{-1},u\beta_1) + \ell(x^{-1},-u\beta_1) = \ell_u(x)-2\ell(x^{-1},u\beta_1).
\end{align*}
We compute
\begin{align*}
\ell(x^{-1},u\beta_1) = \ell(w^{-1}t^{-w\mu},u\beta_1) = \langle -w\mu,u\beta_1\rangle + \Phi^+(u\beta_1) - \Phi^+(w^{-1}u\beta_1) = m + \Phi^+(u\beta_1).
\end{align*}
The claimed dimension formula for $\mathcal T$ follows just as above. This finishes the induction and the proof.
\begin{align*}
\dim\mathcal  T &= \dim\mathcal  T' - m - \Phi^+(u\beta_1) = \frac 12(N-\ell_{us_{\beta_1}}(x)) - m - \Phi^+(u\beta_1)
\\&=\frac 12(N - \ell_u(x) + 2\ell(x^{-1},u\beta_1)) - m - \Phi^+(u\beta_1)
=\frac 12(N - \ell_u(x)).\qedhere
\end{align*}
\end{proof}
We reformulate this proposition to describe arbitrary intersections of semi-infinite orbits.
\begin{theorem}\label{thm:semiInfiniteOrbitsViaTypes}
Let $u,v\in W$ and $x=w_xt^{\mu_x},y\in\widetilde W$. Pick a reflection order $\Phi^+ = \{\beta_1\prec\cdots\prec\beta_{\#\Phi^+}\}$ and an index $n\in\{0,\dotsc,\#\Phi^+\}$ such that
\begin{align*}
u^{-1}v = s_{\beta_{n+1}}\cdots s_{\beta_{\#\Phi^+}}.
\end{align*}
Then we get a decomposition into locally closed subsets
\begin{align*}
\left((\prescript u{}{}U(L) \cap\prescript v{}{}U(L))xI\right)\cap (\prescript {uw_0}{}U(L)yI)=\bigsqcup_\tau x\mathcal T_\tau\subset G(L)/I,
\end{align*}
where $\tau$ runs through all  $\tau = \{(n_1,\nu_1),\dotsc,(n_N,\nu_N)\}$ which are admissible types for $(x^{-1}y,w_x^{-1}u,\prec)$ and satisfy the additional constraint $N=0$ or $n_N\leq n$. Each piece $x \mathcal T_\tau = x \mathcal T_{w_x^{-1}u,\prec,\tau}\subset G(L)/I$ is a locally closed subset of the affine flag variety, and an irreducible smooth affine $k$-scheme of dimension
\begin{align*}
\dim x \mathcal T_\tau = \frac 12\left(\ell_u(x)-\ell_u(y)+\#\tau\right).
\end{align*}
\end{theorem}
\begin{proof}
For all $\beta\in \Phi^+$, we have
\begin{align*}
U_{ u\beta}(L)\subseteq \prescript v{}{}U(L)&\iff v^{-1}u\beta\in \Phi^+\underset{\text{L\ref{lem:reflectionOrderProperties}}}\iff\beta\in\{\beta_1,\dotsc,\beta_n\}.
\end{align*}
Hence
\begin{align*}
&\prescript u{}{}U(L)\cap \prescript {v}{}{}U(L) = U_{u\beta_1}(L)\cdots U_{u\beta_{n}}(L)
\\&x^{-1}(\prescript u{}{}U(L)\cap \prescript {v}{}{}U(L))x = U_{w_x^{-1}u\beta_1}(L)\cdots U_{w_x^{-1}u\beta_{n}}(L).
\end{align*}
By Proposition~\ref{prop:semiInfiniteTypes}, we obtain a decomposition of the corresponding subset of the affine flag variety
\begin{align*}
U_{w_x^{-1}u\beta_1}(L)\cdots U_{w_x^{-1}u\beta_n}(L)I~\subset~G(L)/I
\end{align*}
into types $\{(n_1,\nu_1),\dotsc,(n_N,\nu_N)\}$ with $n_N\leq n$. Denote by $\mathcal T=\mathcal T_{w_x^{-1}u,\prec,\tau}$ the subset associated with such a type $\tau$ as in Lemma~\ref{lem:typeVarieties}. We have \begin{align*}\mathcal T \subset x^{-1}\prescript{uw_0}{}{}U(L)yI = \prescript{w_x^{-1} uw_0}{}{}U(L)x^{-1} yI\end{align*}
if and only if $\tau$ is admissible for $(x^{-1}y,w_x^{-1}u,\prec)$. It remains to compute the dimension of $\mathcal T$ using Proposition~\ref{prop:Tgeom} and \cite[Lemma~2.12]{Schremmer2022_newton}:
\begin{align*}
2\dim \mathcal T =& N - \ell_{w_x^{-1} u}(x^{-1} y) \\=& N - \ell_{w_x^{-1} u}(x^{-1}) - \ell_u(y) = N + \ell_u(x)-\ell_u(y).\qedhere
\end{align*}
\end{proof}
\begin{remark}\begin{enumerate}[(a)]
\item It seems reasonable to expect that each type variety $\mathcal T_{u,\prec,\tau}$ should be a product of affine lines and pointed affine lines over $k$, but the proof of such a statement would probably require undue analysis of the polynomials $f^{(\bullet)}_\bullet$ or some major progress towards Zariski's cancellation problem. We don't need such a precise description.
\item Given $u,v$, there are in general several possible reflection orders satisfying $v^{-1} u =s_{\beta_{n+1}}\cdots s_{\beta_{\#\Phi^+}}$ for $n = \ell(w_0 v^{-1} u)$. While the geometry of the intersection
\begin{align*}
\left((\prescript u{}{}U(L) \cap\prescript v{}{}U(L))xI\right)\cap (\prescript {uw_0}{}U(L)yI)
\end{align*}
does not depend on the choice of reflection order, the decomposition into subsets indexed by types tends to do that, i.e.\ different reflection orders yield different subsets. It is not clear how these are related, aside from the simple observation that subsets of maximal dimension parametrize irreducible components of maximal dimension. We will prove that the \emph{number} of such subsets $x\mathcal T_\tau$ of any given dimension is independent of the chosen reflection order.
\item For each given $(x,u,\prec)$, there exist only finitely many admissible types. This is straightforward to prove directly, and will immediately follow from a later result (cf.\ Lemma~\ref{lem:typesAndPaths}). So Theorem~\ref{thm:semiInfiniteOrbitsViaTypes} provides a decomposition into finitely many locally closed pieces.
\end{enumerate}
\end{remark}
\begin{example}
Consider the group $G = \GL_3$. Let $T$ be the torus of diagonal matrices and $B$ be the upper triangular matrices. Let $u=v=w_0$ and 
\begin{align*}
x = w_0t^{\rho^\vee} = \begin{pmatrix}0&0&t^{-1}\\0&1&0\\t&0&0\end{pmatrix}.
\end{align*}
We denote our simple roots by $\Delta = \{\alpha,\beta\}$ corresponding to the diagonal matrices $\dot\alpha = \diag(t,t^{-1},0)$ and $\dot\beta = (0,t,t^{-1})$. We choose the reflection order $\alpha\prec\alpha+\beta\prec\beta$. Then the admissible types for $(x,u,\prec)$ are given by the cardinality $1$ type $\{(-\alpha-\beta,-1)\}$ as well as the cardinality $3$ types
\begin{align*}
\{(-\beta,0),(-\alpha-\beta,-1),(-\alpha,0)\},\quad \{(-\alpha,-1),(-\alpha-\beta,0),(-\beta,-1)\}.
\end{align*}
We see that the intersection
\begin{align*}
\prescript{w_0}{}{}U(L)I~\cap~U(L)xI~\subset~G(L)/I
\end{align*}
has dimension
\begin{align*}
\frac 12\left(3-\ell_{w_0}(x)\right) = \frac 12\left(3-\langle -\rho^\vee,2\rho\rangle +\ell(w_0)\right) = \frac 12(3+4+3) = 5,
\end{align*}
and the number of $5$-dimensional irreducible components is two.

\end{example}
\section{Double Bruhat graph}\label{sec:DBG}
There is a more convenient and natural way to encode the datum of an admissible type $(n_1,\nu_1),\dotsc,(n_N,\nu_N)$. This construction is due to Naito--Watanabe \cite[Section~5.1]{Naito2017}, used originally to describe periodic $R$-polynomials of affine Weyl groups. 
\begin{definition}
Let $\Phi^+ = \{\beta_1\prec\cdots\prec\beta_{\#\Phi^+}\}$ be a reflection order and $v,w\in W$.
\begin{enumerate}[(a)]
\item The \emph{double Bruhat graph} $\DBG(W)$ is a finite directed graph. Its set of vertices is $W$. For each $w\in W$ and $\alpha\in \Phi^+$, there is an edge $w\xrightarrow{\alpha} ws_\alpha$.
\item A \emph{non-labelled path} $\overline p$ in $\DBG(W)$ is a sequence of adjacent edges
\begin{align*}
\overline p : v = u_1\xrightarrow{\alpha_1} u_2\xrightarrow{\alpha_2} \cdots\xrightarrow{\alpha_\ell} u_{\ell+1} = w.
\end{align*}
We call $\overline p$ an unlabelled path from $v$ to $w$ of length $\ell(\overline p) = \ell$. We say $\overline p$ is \emph{increasing} with respect to $\prec$ if $\alpha_1\prec\cdots\prec\alpha_\ell$. We say that $\overline p$ \emph{is bounded by }$n\in\mathbb Z$ if each occurring root $\alpha_i$ has the form $\alpha_i=\beta_j$ for $j\leq n$.
\item A \emph{labelled path} or \emph{path} $p$ in $\DBG(W)$ consists of an unlabelled path
\begin{align*}
\overline p: v = u_1\xrightarrow{\alpha_1}u_2\xrightarrow{\alpha_2}\cdots\xrightarrow{\alpha_\ell} u_{\ell+1}=w
\end{align*}
together with integers $m_1,\dotsc,m_\ell\in\mathbb Z$ subject to the condition
\begin{align*}
m_i\geq \Phi^+(-u_i\alpha_i)=\begin{cases}0,&\ell(u_{i+1})>\ell(u_i),\\
1,&\ell(u_{i+1})<\ell(u_i).\end{cases}
\end{align*}
We write $p$ as
\begin{align*}
p : v = u_1\xrightarrow{(\alpha_1,m_1)}u_2\xrightarrow{(\alpha_2,m_2)}\cdots\xrightarrow{(\alpha_\ell,m_\ell)} u_{\ell+1}=w.
\end{align*}
The \emph{weight} of $p$ is
\begin{align*}
\wt(p) = m_1\alpha_1^\vee+\cdots+m_\ell\alpha_\ell^\vee\in\mathbb Z\Phi^\vee.
\end{align*}
The \emph{length} of $p$ is $\ell(p) = \ell(\overline p) = \ell$. We say that $p$ is \emph{increasing} with respect to $\prec$ if $\overline p$ is. We say that $p$ is bounded by $n\in\mathbb Z$ if $\overline p$ is.
\item The set of all paths from $v$ to $w$ that are increasing with respect to $\prec$ and bounded by $n\in\mathbb Z$ is denoted $\paths^\prec_{\preceq n}(v\Rightarrow w)$. We also write \begin{align*}\paths^\prec(v\Rightarrow w) := \paths^\prec_{\preceq\#\Phi^+}(v\Rightarrow w)\end{align*}
for the set of all increasing paths from $v$ to $w$.
\end{enumerate}
\end{definition}
\begin{example}
This is the double Bruhat graph of type $A_2$, where we denote the simple roots by $\Delta = \{\alpha_1, \alpha_2\}$ and the corresponding simple reflections by $S = \{s_1, s_2\}$. For each root $\alpha \in \Phi^+ = \{\alpha_1, \alpha_2,\alpha_1+\alpha_2\}$ and each $w\in W$, there is an edge $w\rightarrow ws_\alpha$ with label $\alpha$ and the converse edge $ws_\alpha\rightarrow w$ with the same label, making each edge appear doubled (which explains the graph's name).
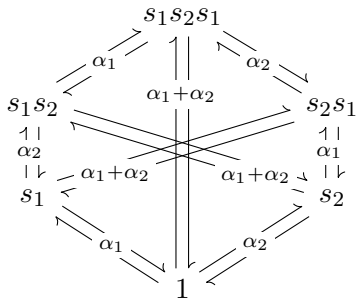
\begin{align*}
\begin{tikzcd}[ampersand replacement=\&,column sep=2em]
\& s_1 s_2 s_1\ar[dl, harpoon, shift left=0.5ex]\ar[dr, harpoon, shift left=0.5ex]
\ar[ddd, harpoon, shift left=0.5ex]
\& \\ 
s_1 s_2\ar[ru,"{\alpha_1}" {description,xshift=0.5ex,yshift=-.5ex}, harpoon, shift left=0.5ex]\ar[d, harpoon, shift left=0.5ex]\ar[rrd, harpoon, shift left=0.5ex]
\& \& s_2 s_1\ar[lu, "{\alpha_2}" {description,xshift=-0.01ex,yshift=0.3ex}, harpoon, shift left=0.5ex]\ar[d,shift left=0.5ex, harpoon]\ar[dll, shift left=0.5ex, harpoon]
\\ 
s_1\ar[u, "{\alpha_2}" {description, xshift=0.3ex}, harpoon, shift left=0.5ex]\ar[rru, "{\alpha_1 + \alpha_2}" {description,near start,yshift=-.35em}, harpoon, shift left=0.5ex]\ar[dr,harpoon,shift left=0.5ex]
\& \& s_2\ar[u, "{\alpha_1}" {description, xshift=0.3ex}, harpoon, shift left=0.5ex]\ar[llu, "{\alpha_1 + \alpha_2}" {description,near start,xshift=.35em}, harpoon, shift left=0.5ex]\ar[dl, shift left=0.5ex, harpoon]
\\ 
\& 1\ar[ru, "{\alpha_2}" {description, xshift=0.4ex,yshift=-0.4ex}, harpoon, shift left=0.5ex]\ar[lu, "{\alpha_1}" {description, xshift=0.4ex,yshift=0.3ex}, harpoon, shift left=0.5ex]\ar[uuu, "{\alpha_1 + \alpha_2}" {description, near end, xshift=0.4ex}, harpoon, shift left=0.5ex]
\end{tikzcd}
\end{align*}
The two reflection orders are given by
\begin{align*}
&\alpha_1\prec \alpha_1+\alpha_2\prec \alpha_2,\\
&\alpha_2\prec \alpha_1+\alpha_2\prec \alpha_1.
\end{align*}
\end{example}
\begin{remark}
One could similarly study paths in the double Bruhat graph bounded from below by a root $\beta_m$, or even paths such that $\alpha_1,\dotsc, \alpha_{\ell(p)} \in\{\beta_m,\dotsc,\beta_n\}$ for fixed indices $m,n\in\mathbb Z$. However, this extra generality would not yield any new information:

If $w_0 = s_{\alpha_1}\cdots s_{\alpha_{\#\Phi^+}}$ is the reduced expression corresponding to our reflection order and $0\leq m\leq n\leq \#\Phi^+$, we could consider a different reflection order $\prec'$ coming from the reduced word
\begin{align*}
w_0 = s_{\alpha_{m+1}}\cdots s_{\alpha_{\#\Phi^+}} s_{-w_0\alpha_1}\cdots s_{-w_0\alpha_{m}}.
\end{align*}
If we write $\Phi^+ = \{\beta'_1\prec'\cdots\prec'\beta'_{\#\Phi^+}\}$, then
\begin{align*}
\beta_i = s_{\alpha_1}\cdots s_{\alpha_{m}} \beta'_{i-m},\qquad i = m+1,\dotsc,n.
\end{align*}
This would give a one-to-one correspondence of paths which are increasing for $\prec$ and where all roots lie \emph{between} $m$ and $n$ and paths which are increasing for $\prec'$ and bounded by $n-m$.
\end{remark}
The comparison between types and paths is given as follows.
\begin{lemma}\label{lem:typesAndPaths}
Let $\Phi^+ = \{\beta_1\prec\cdots\prec\beta_{\#\Phi^+}\}$ be a reflection order, $x=w t^\mu\in\widetilde W$ and $u\in W$. Let moreover $n\in\{0,\dotsc,\#\Phi^+\}$. 

Let $A$ be the set of all admissible types $\{(n_1,\nu_1),\dotsc,(n_N,\nu_N)\}$ for $(x,u,\prec)$ satisfying the condition $n_1<\cdots<n_N\leq n$.
Define
\begin{align*}
P := \{p\in\paths^\prec_{\preceq n}(w^{-1} u\Rightarrow u)\mid \wt(p) = u^{-1} w\mu\}.
\end{align*}
Then we get a bijective map $A\xrightarrow\sim P$ sending a type $\tau = \{(n_1,\nu_1),\dotsc,(n_N,\nu_N)\}$ of cardinality $N$ to the length $N$ path
\begin{align*}
p : w^{-1} u\xrightarrow{(\beta_{n_1}, m'_1)}\cdots \xrightarrow{(\beta_{n_N},m'_N)} u
\end{align*}
in $P$, where $(\beta'_h,m'_h) := r_{b_N}\cdots r_{b_{h}}(b_h)\in \Phi_\af^+$ and $b_h = (u\beta_{n_h},\nu_h)\in \Phi_\af$ for $h=1,\dotsc,N$.
\end{lemma}
\begin{remark}
We remark that if $\mu\notin \mathbb Z\Phi^\vee$ in the above lemma, then one trivially gets $P = \emptyset$ by definition of the double Bruhat graph. Moreover, it is clear that such an $x$ can not be a product of affine reflections, hence trivially $A=\emptyset$ by definition of admissible types. The statement is more interesting when $\mu\in \mathbb Z\Phi^\vee$, that is, when $x$ lies in the non-extended affine Weyl group.
\end{remark}
\begin{proof}[Proof of Lemma~\ref{lem:typesAndPaths}]
For now, we fix arbitrary integers $1\leq n_1<\cdots<n_N\leq n$ and $\nu_1,\dotsc,\nu_N\in\mathbb Z$. Define the set \begin{align*}
\tau = \{(n_1,\nu_1),\dotsc,(n_N,\nu_N)\}\subset\mathbb Z\times\mathbb Z.
\end{align*}
For $h\in\{1,\dotsc,N\}$, define $b_h := (u\beta_{n_h}, \nu_h)\in\Phi_\af$ and
\begin{align*}
b'_h = (\beta'_h,m'_h) := r_{b_N}\cdots r_{b_{h}}(b_h)\in \Phi_\af.
\end{align*}

Finally, we write down something that may or may not be a path in $\DBG(W)$ as
\begin{align*}
p : u s_{\beta_{n_N}}\cdots s_{\beta_{n_1}}\xrightarrow{(\beta_{n_1},m'_1)}\cdots\xrightarrow{(\beta_{n_N},m'_N)} u.
\end{align*}
We want to show that $\tau\in A$ if and only if $p\in P$, since the desired bijection is immediate from this.

The key observation is for $h\in\{1,\dotsc,N\}$ that
\begin{align*}
r_{b_N}\cdots r_{b_{h+1}}(b_h) \in \Phi_\af^-\iff&
r_{b_N}\cdots r_{b_{h}}(b_h) \in \Phi_\af^+\\\iff&
\beta'_h \in \Phi_\af^+\\\iff&
\Bigl(u s_{\beta_{n_N}}\cdots s_{\beta_{n_h}} \beta_{n_h}, m'_h\Bigr)\in \Phi_\af^+
\\\iff&m'_h\geq \Phi^+(u s_{\beta_{n_N}}\cdots s_{\beta_{n_h}}(\beta_{n_h})).
\end{align*}
Thus $\tau$ is an admissible type for some $(\tilde x, u,\prec)$ if and only if $p$ is a well-defined path in the double Bruhat graph. In this case, it is clear that $p$ is increasing with respect to $\prec$ and bounded by $n$.

We calculate
\begin{align*}
r_{b_1}\cdots r_{b_N} =& r_{b'_N}\cdots r_{b'_1} = s_{\beta_N'}\cdots s_{\beta_1'}t^{ m_N' s_{\beta_{1}'}\cdots s_{\beta_{N-1}'} (\beta_N')^{\vee}+\cdots +m_1'(\beta_1')^\vee}.
\end{align*}
In case
\begin{align*}
s_{u\beta_{n_1}}\cdots s_{u\beta_{n_N}} \neq w,
\end{align*}
we see that $\tau$ cannot be admissible for $(x,u,\prec)$ and $p$ does not start in $w^{-1} u$. So assume from now on that
\begin{align*}
s_{u\beta_{n_1}}\cdots s_{u\beta_{n_N}} = w.
\end{align*}
Then we can simplify
\begin{align*}
r_{b_1}\cdots r_{b_N} =& \ldots=wt^{ m_N' s_{\beta_{1}'}\cdots s_{\beta_{N-1}'} (\beta_N')^{\vee}+\cdots +m_1'(\beta_1')^\vee}.
\end{align*}
We compute for $h=1,\dotsc,N$ that
\begin{align*}
&m_h' s_{\beta_1'}\cdots s_{\beta_{h-1}'}(\beta_h') = m_h' (s_{\beta_1'}\cdots s_{\beta_N'}) (s_{\beta_{N}'}\cdots s_{\beta_h'})\beta_h'^\vee
\\=&m_h' w^{-1} s_{\beta_{N}'}\cdots s_{\beta_h'}\beta_h'^\vee
=m_h' w^{-1} u \beta_h^\vee.
\end{align*}
We conclude
\begin{align*}
r_{b_1}\cdots r_{b_N} = w t^{w^{-1}u\wt(p)}.
\end{align*}
Hence $\tau$ is admissible for $(x,u,\prec)$ if and only if $\wt(p) = u^{-1}w\mu$.
\end{proof}
We can state the double Bruhat version of Theorem~\ref{thm:semiInfiniteOrbitsViaTypes} as follows. It can be seen as an analogue of \cite[Theorem~7.1]{Parkinson2009}. This result studies slightly different intersections in the affine flag variety and uses folded alcove walks rather than paths in the double Bruhat graph.
\begin{theorem}\label{thm:semiInfiniteOrbitsViaPaths}
Let $u,v\in W$ and $x=w_x t^{\mu_x},y = w_yt^{\mu_y}\in\widetilde W$. Pick a reflection order $\Phi^+ = \{\beta_1\prec\cdots\prec\beta_{\#\Phi^+}\}$ and an index $n\in\{0,\dotsc,\#\Phi^+\}$ such that
\begin{align*}
u^{-1}v = s_{\beta_{n+1}}\cdots s_{\beta_{\#\Phi^+}}.
\end{align*}
Then we get a decomposition into locally closed subsets
\begin{align*}
\left((\prescript u{}{}U(L) \cap\prescript v{}{}U(L))xI\right)\cap (\prescript {uw_0}{}U(L)yI)=\bigsqcup_p x\mathcal T_p,
\end{align*}
where $p$ runs through all paths $p\in\paths^\prec_{\preceq n}(w_y^{-1}u\Rightarrow w_x^{-1}u)$ such that \begin{align*}
u\wt(p) = w_y\mu_y - w_x\mu_x.\end{align*} Each variety $x\mathcal T_p\subseteq G(L)/I$ is an irreducible smooth affine $k$-scheme of dimension
\begin{align*}
&\dim (x\mathcal T_p) = \dim \mathcal T_p = \frac 12\left(\ell_u(x)-\ell_u(y)+\ell(p)\right).\rightqed
\end{align*}
\end{theorem}
The aim of this section is to develop the basic properties of the double Bruhat graph.
For the remainder of this section, fix a reflection order $\Phi^+ = \{\beta_1\prec\cdots\prec\beta_{\#\Phi^+}\}$. We first introduce a couple of immediate properties about paths in $\paths^\prec_{\preceq\cdot}(\cdot\Rightarrow \cdot)$.
Part (a) of the following lemma is an adaption of the duality anti-automorphism from \cite[Proposition~4.3]{Lenart2015}.
\begin{lemma}\label{lem:pathSymmetries}
Let $u,v\in W$ and $n\in\mathbb Z$.
\begin{enumerate}[(a)]
\item Denote by $\succ$ the reflection order obtained by reversing $\prec$. Then we have a bijection
\begin{align*}
\paths^\prec_{\preceq n}(u\Rightarrow v)\rightarrow& \paths^\succ_{\succeq n}(w_0 v\Rightarrow w_0 u)\\
\left(w_1\xrightarrow{(\alpha_1,n_1)}\cdots\xrightarrow{(\alpha_\ell,n_\ell)}w_{\ell+1}\right)\mapsto&\left(w_0 w_{\ell+1}\xrightarrow{(\alpha_\ell, n_\ell)}\cdots\xrightarrow{(\alpha_1,n_1)}w_0 w_1\right)
\end{align*}
that preserves both the weight and the length of each path.
\item Let $\prec'$ be the reflection order defined by
\begin{align*}
\alpha\prec'\beta\iff -w_0\alpha\prec-w_0\beta.
\end{align*}
Then we a bijection
\begin{align*}
\paths^\prec_{\preceq n}(u\Rightarrow v)\rightarrow& \paths^{\prec'}_{\preceq'n}(w_0 uw_0\Rightarrow w_0 vw_0)\\
\left(w_1\xrightarrow{(\alpha_1,n_1)}\cdots\xrightarrow{(\alpha_\ell,n_\ell)}w_{\ell+1}\right)\mapsto&\left(w_0 w_1w_0 \xrightarrow{(-w_0 \alpha_1,n_1)}\cdots\xrightarrow{(-w_0 \alpha_\ell,n_\ell)}w_0 w_{\ell+1}w_0\right)
\end{align*}
that preserves the lengths of paths. This bijection sends path of weight $\omega$ to a path of weight $-w_0 \omega$.
\item Let $x\in \Omega$ be of length zero in $\widetilde W$ and write it as $x = wt^\mu$. Then we have a bijection
\begin{align*}
\paths^\prec_{\preceq n}(u\Rightarrow v)\rightarrow &
\paths^\prec_{\preceq n}(wu\Rightarrow wv)\\
\left(w_1\xrightarrow{(\alpha_1,n_1)}\cdots\xrightarrow{(\alpha_\ell,n_\ell)}w_{\ell+1}\right)\mapsto&\left(w w_1\xrightarrow{(\alpha_1, n_1-\langle \mu, w_1 \alpha_1\rangle)}\cdots\xrightarrow{(\alpha_\ell, n_\ell-\langle\mu,w_\ell\alpha_\ell\rangle)} w w_{\ell+1}\right)
\end{align*}
that preserves the lengths of paths. This bijection sends a path of weight $\omega$ to a path of weight $\omega + v^{-1}\mu - u^{-1}\mu$.
\end{enumerate}
\end{lemma}
\begin{proof}
Straightforward verification.
\end{proof}
It is rather natural and very fruitful to compare the definition of the double Bruhat graph with the similar concept of the \emph{quantum Bruhat graph}. The quantum Bruhat graph can be defined as a subgraph of the double Bruhat graph, containing all those paths
\begin{align*}
p : u_1\xrightarrow{(\alpha_1,m_1)}\cdots\xrightarrow{(\alpha_\ell,m_\ell)} u_{\ell+1}
\end{align*}
satisfying the additional constraint that for each $n\in\{1,\dotsc,m_\ell\}$ either
\begin{itemize}
\item $\ell(u_{n+1}) = \ell(u_n)+1$ and $m_n=0$ or
\item $\ell(u_{n+1}) = \ell(u_n)+1-\langle \alpha_n^{\vee},2\rho\rangle$ and $m_n=1$.
\end{itemize}
The quantum Bruhat graph has been introduced by Brenti-Fomin-Postnikov \cite{Brenti1998} in order to study certain solutions to the Yang--Baxter equations related to the quantum Chevalley-Monk formula. It has since occurred frequently in literature on quantum cohomology of flag varieties, e.g.\ in \cite{Postnikov2005}. Due to its relationship to the affine Bruhat order discovered by Lam-Shimozono \cite{Lam2010}, it since has played a major role in the study of affine Bruhat order and affine Deligne--Lusztig varieties, e.g.\ in \cite{Milicevic2021, Milicevic2020, He2021d, Schremmer2024_bruhat}.

The initial article \cite{Brenti1998} is of special interest to us, since we may identify its \emph{main result} as a crucial statement on the double Bruhat graph. The authors derive the fundamental properties of the quantum Bruhat graph as an application of their main result, which we may recognize as the aforementioned embedding of the quantum Bruhat graph into the double Bruhat graph.

Let us recall the main result of \cite{Brenti1998} in the authors' language. They construct a family of solutions to the Yang--Baxter equations for the finite Weyl group $W$ as follows: Choose a field $K$ of characteristic zero, and \emph{multiplicative functions} $E_1, E_2:\Phi^+\rightarrow K$, i.e.\ functions satisfying
\begin{align*}
E_i(\alpha+\beta)=E_i(\alpha)E_i(\beta)\text{ whenever } \alpha,\beta,\alpha+\beta\in\Phi^+.
\end{align*}
We have to assume that there is no root $\alpha\in\Phi^+$ such that $E_1(\alpha)=E_2(\alpha)=0$.
Choose moreover constants $\kappa_\alpha\in K$ depending only on the length of $\alpha\in\Phi^+$. Define for each $\alpha\in \Phi^+$ the following $K$-linear endomorphism of the group algebra $K[W]$:
\begin{align*}
R_\alpha : K[W]\rightarrow K[W],w\mapsto \begin{cases}w+p_\alpha s_\alpha w,&\ell(s_\alpha w)>\ell(w),\\
w+q_\alpha s_\alpha w,&\ell(s_\alpha w)<\ell(w).\end{cases}
\end{align*}
Here, we define the scalars
\begin{align*}
p_\alpha = \frac{\kappa_\alpha E_1(\alpha)}{E_1(\alpha)-E_2(\alpha)},\quad q_\alpha = \frac{\kappa_\alpha E_2(\alpha)}{E_1(\alpha)-E_2(\alpha)}\in K.
\end{align*}
Then the linear functions $\{R_\alpha\}_{\alpha\in \Phi^+}$ satisfy the \emph{Yang--Baxter equations}. We do not wish to recall how these equations are defined, referring the reader to the original article of Brenti-Fomin-Postnikov for the details. We do want to note, however, the following consequence of the Yang--Baxter equations, which is proved completely analogously to \cite[Proposition~2.5]{Brenti1998}:
\begin{proposition}\label{prop:ybsolution}
Let $\Phi^+ = \{\beta_1\prec\cdots\prec\beta_{\#\Phi^+}\} = \{\gamma_1\prec'\cdots\prec'\gamma_{\#\Phi^+}\}$ be two reflection orders on $\Phi^+$ and $0\leq n\leq \#\Phi^+$ such that
\begin{align*}
&s_{\beta_1}\cdots s_{\beta_n} = s_{\gamma_1}\cdots s_{\gamma_n}.
\end{align*}
Then
\begin{align*}
R_{\beta_1}\cdots R_{\beta_n} = R_{\gamma_1}\cdots R_{\gamma_n}
\end{align*}
as endomorphisms on $K[W]$.\rightqed
\end{proposition}
How is this related to the double Bruhat graph? We can make the following choices, which are essentially universal\footnote{We may focus on irreducible root systems, where only one of the two functions $E_1, E_2$ is allowed to vanish on the highest root. If this is WLOG $E_1$, replacing $(E_1, E_2)$ by $(1, E_2/E_1)$ yields the same solutions $R_\bullet$. Now a universal solution would be given by choosing $E_1(\alpha) = e^{\alpha}$ and $\kappa_\alpha \in\{\kappa_s,\kappa_l\}$ for the subring $R$ contained in our choice of $K$ generated by these values and the inverses $1/(E_1(\alpha)-1)$. However, $R$ is not a field, and we would like to expand the geometric series.}:

Write $\Delta = \{\alpha_1,\dotsc,\alpha_{\#\Delta}\}$ and let $K$ be the field of formal Laurent series over $\mathbb Q$ in $2+{\#\Delta}$ formal variables, denoted
\begin{align*}
K = \mathbb Q\doubleparen{\kappa_s,\kappa_l,e^{\alpha_1},\dotsc,e^{\alpha_{\#\Delta}}}.
\end{align*}
To an element of the root lattice $\lambda =c_1\alpha_1+\cdots + c_{\#\Delta}\alpha_{\#\Delta}$, we asssocite the element $e^\lambda = (e^{\alpha_1})^{c_1}\cdots (e^{\alpha_{\#\Delta}})^{c_{\#\Delta}}\in K$. Define $E_1(\alpha) = 1$ and $E_2(\alpha)=e^\alpha\in K$ for all $\alpha\in\Phi^+$. We put $\kappa_\alpha = \kappa_s$ if $\alpha$ is short and $\kappa_\alpha = \kappa_l$ if $\alpha$ is long. Then we observe
\begin{align*}
R_\alpha(w) = w + \frac{\kappa_\alpha e^{\Phi^+(-w^{-1}\alpha)\alpha}}{1-e^\alpha}s_\alpha w = w + \sum_{i\geq \Phi^+(-w^{-1}\alpha)} \kappa_\alpha e^{i\alpha}s_\alpha w.
\end{align*}
This basically describes all paths from $w^{-1}$ to $w^{-1}$ or $w^{-1}s_\alpha$ in the double Bruhat graph associated to the dual root system, with the restriction that the only occurring edge may be $w^{-1}\xrightarrow{\alpha^\vee} w^{-1}s_\alpha$. Composition of the linear operators $R_\bullet$ along a reflection order recovers precisely our notion of paths in the double Bruhat graph which are increasing with respect to that reflection order.
\begin{corollary}\label{cor:reflectionOrderInvariance}
Let $\Phi^+ = \{\beta_1\prec\cdots\prec\beta_{\#\Phi^+}\} = \{\gamma_1\prec'\cdots\prec'\gamma_{\#\Phi^+}\}$ be two reflection orders on $\Phi^+$ and $0\leq n \leq \#\Phi^+$ such that
\begin{align*}
&s_{\beta_1}\cdots s_{\beta_n} = s_{\gamma_1}\cdots s_{\gamma_n}.
\end{align*}
Let $u,v\in W$, $\mu\in\mathbb Z\Phi^\vee$ and $\ell_s, \ell_l\in\mathbb Z_{\geq 0}$.

Let $p_\prec$ be the number of paths $p\in\paths^\prec_{\preceq n}(u\Rightarrow v)$ such that $\wt(p) = \mu$, $\ell(p) = \ell_s+\ell_l$ and the number of short (resp.\ long) roots occurring as labels in $p$ is equal to $\ell_s$ (resp.\ $\ell_l$). Define $p_{\prec'}$ similarly. Then $p_\prec=p_{\prec'}$.
\end{corollary}
\begin{proof}
Construct the operators $R_\bullet^\vee$ as above for the dual root system $\Phi^\vee$, and consider the $\mathbb Q$-coefficient of $\kappa_l^{\ell_s}\kappa_s^{\ell_l}e^\mu v^{-1}$ in the expansion of
\begin{align*}
R_{\beta_{n}^\vee}^\vee\cdots R_{\beta_1^\vee}^\vee u^{-1}\in K[W].
\end{align*}
By construction, this number is equal to $p_\prec$. By Proposition~\ref{prop:ybsolution}, it is also equal to $p_{\prec'}$.
\end{proof}
It is convenient to use the language of \emph{multisets} when keeping track of the lengths and weights of paths. Recall that a multiset is a modification of the concept of a set, where elements are allowed to be contained multiple times in a multiset.

Formally, we may define a multiset $M$ as a tuple $(\abs M, m)$ where $\abs M$ is any set and $m : \abs M \rightarrow\mathbb Z_{\geq 1}\cup\{+\infty\}$ is a function (to be thought of as counting how often an element occurs in $M$). We write $x\in M$ meaning $x\in \abs M$, and say that $x$ has multiplicity $m(x)$ in $M$. If $x\notin M$, we say that $x$ has multiplicity zero in $M$.

If $f$ is a map from $\abs M$ to some abelian group (e.g.\ the real numbers), we write
\begin{align*}
\sum_{x\in M} f(x):=\sum_{x\in \abs M} m(x) f(x),
\end{align*}
meaning that elements are summed with multiplicity (depending on the function and the abelian group, such a sum may or may not be well-defined). The cardinality of $M$ can then be defined as
\begin{align*}
\#M := \sum_{x\in M}1\in\mathbb Z_{\geq 0}\cup\{\infty\}.
\end{align*}

If $M, M'$ are two multi-sets, we define their additive union $M\cup M'$ by declaring that $x$ has multiplicity $m_1+m_2$ in $M\cup M'$ where $m_1\in \mathbb Z_{\geq 0}\cup \{+\infty\}$ is the multiplicity of $x$ in $M$ and $m_2$ the multiplicity of $x$ in $M'$.

When explicitly writing down multi-sets via a list of elements (the number of occurrences expressing the multiplicity), we use the notation $\{\cdot\}_m$ to distinguish from the usual set notation $\{\cdot\}$.
\begin{definition}\label{def:weightMultiset}
Let $\Phi^+ = \{\beta_1\prec\cdots\prec\beta_{\#\Phi^+}\}$ be a reflection order and $n\in\{0,\dotsc,\#\Phi^+\}$.
\begin{enumerate}[(a)]
\item We write
\begin{align*}
\pi_{\succ n} = s_{\beta_{n+1}}\cdots s_{\beta_{\#\Phi^+}}\in W.
\end{align*}
\item If $u,v\in W$, we denote
\begin{align*}
\wts(u\Rightarrow v\dashrightarrow v\pi_{\succ n})
\end{align*}
to be the multiset
\begin{align*}
\Bigl\{(\wt(p),\ell(p))\mid p\in \paths^\prec_{\preceq n}(u\Rightarrow v)\Bigr\}_m,
\end{align*}
i.e.\ the multiplicity of $(\omega,e)\in \wts(u\Rightarrow v\dashrightarrow v\pi_{\succ n})$ is equal to the number of paths in $\paths^\prec_{\preceq n}(u\Rightarrow v)$ of weight $\omega$ and length $e$.

We use the shorthand notation
\begin{align*}
\wts(u\Rightarrow v):=&\wts(u\Rightarrow v\dashrightarrow v).
\end{align*}
\end{enumerate}
\end{definition}
The reflection order does not occur any more in the $\wts(\cdots)$-notation, due to Corollary~\ref{cor:reflectionOrderInvariance} (and the usual observation $w_0 = s_{\beta_1}\cdots s_{\beta_{\#\Phi^+}}$). From Lemma~\ref{lem:reflectionOrderProperties}, we see that $\ell(\pi_{\succ n}) = \#\Phi^+ - n$ and that for each $u\in W$, one may find a suitable reflection order $\prec$ with $u = \pi_{\succ \#\Phi^+ - \ell(u)}$. Hence the notation $\wts(u\Rightarrow v\dashrightarrow w)$ is well-defined for all $u,v,w\in W$.

We note the following immediate properties.
\begin{lemma}\label{lem:qbgNonEmptiness}
Let $u,v,v'\in W$. Then the multiset $\wts(u\Rightarrow v\dashrightarrow v')$ is non-empty if and only if the following inequality on the Bruhat order of $W$ is satisfied:
\begin{align*}
v^{-1} v' \leq u^{-1} v'.
\end{align*}
In this case, we have
\begin{align*}
\max \{e\mid (\omega,e)\in \wts(u\Rightarrow v\dashrightarrow v')\} = \ell(u^{-1} v')-\ell(v^{-1} v').
\end{align*}
\end{lemma}
\begin{proof}
Let us pick a reflection order $\Phi^+ = \{\beta_1\prec\cdots\prec\beta_{\#\Phi^+}\}$ with $v' = v\pi_{\succ n}$. Write the corresponding reduced word as $w_0 = s_{\alpha_1}\cdots s_{\alpha_{\#\Phi^+}}$.

The multiset in question is non-empty if there is a sequence $1\leq  i_1<\cdots<i_e\leq n$ of indices such that
\begin{align*}
&v = u s_{\beta_{i_1}}\cdots s_{\beta_{i_e}}
\\\iff&vs_{\alpha_1}\cdots s_{\alpha_n} = u s_{\alpha_1}\cdots\widehat{s_{\alpha_{i_1}}}\cdots \widehat{s_{\alpha_{i_e}}}\cdots s_{\alpha_n}
\\\iff&u^{-1} v s_{\alpha_1}\cdots s_{\alpha_n} = s_{\alpha_1}\cdots \widehat{s_{\alpha_{i_1}}}\cdots \widehat{s_{\alpha_{i_e}}}\cdots s_{\alpha_n}.
\end{align*}
Using the subword criterion for Bruhat order, the existence of such indices $i_1,\dotsc,i_e$ is equivalent to the Bruhat order inequality
\begin{align*}
u^{-1} v s_{\alpha_1}\cdots s_{\alpha_n} \leq s_{\alpha_1}\cdots s_{\alpha_n},
\end{align*}
and the maximal number $e$ is equal to the difference in lengths of the two sides. Now we compute
\begin{align*}
s_{\alpha_1}\cdots s_{\alpha_n} = s_{\beta_{n+1}}\cdots s_{\beta_{\#\Phi^+}} s_{\alpha_1}\cdots s_{\alpha_{\#\Phi^+}} = \pi_{\succ n} w_0.
\end{align*}
Hence the above Bruhat order condition becomes $u^{-1} v' w_0\leq v^{-1} v' w_0$.
The claim follows using the Bruhat order anti-automorphism induced by multiplication by $w_0$.
\end{proof}
\begin{remark}
We can use these multisets to summarize Theorem~\ref{thm:semiInfiniteOrbitsViaPaths} as follows: The number of pieces $\mathcal T_p$ of dimension $d$ occurring in
\begin{align*}
\Bigl((\prescript u{}{}U(L)\cap \prescript v{}{}U(L))xI\Bigr)\cap (\prescript{uw_0}{}{}U(L)yI)
\end{align*}
is equal to the multiplicity of the tuple
\begin{align*}
(u^{-1} w_y \mu_y - u^{-1} w_x\mu_x, 2d-\ell_u(x)+\ell_u(y))
\end{align*}
in the multiset
\begin{align*}
\wts(w_y^{-1}u\Rightarrow w_x^{-1} u\dashrightarrow w_x^{-1} v).
\end{align*}
\end{remark}
We finish this section by comparing the double Bruhat graph more directly to the quantum Bruhat graph.
\begin{proposition}\label{prop:qbgVsDbg}
Let $u,v\in W$. Denote by $d(u\Rightarrow v)$ the distance of a shortest path in the quantum Bruhat graph from $u$ to $v$, and by $\wt(u\Rightarrow v)$ the weight of such a path. Let $(\omega,e)\in \wts(u\Rightarrow v)$.
\begin{enumerate}[(a)]
\item We have $\omega\geq \wt(u\Rightarrow v)$.
\item We have
\begin{align*}
e\leq \langle\omega,2\rho\rangle + \ell(v)-\ell(u).
\end{align*}
If the equality holds, then $\omega = \wt(u\Rightarrow v)$ and $e = d(u\Rightarrow v)$.
\item The multiplicity of
\begin{align*}
(\wt(u\Rightarrow v),d(u\Rightarrow v))
\end{align*}
in the multiset $\wts(u\Rightarrow v)$ is equal to $1$.
\end{enumerate}
\end{proposition}
\begin{proof}
Pick a reflection order $\prec$, and a path $p\in \paths^\prec(u\Rightarrow v)$ of weight $\omega$ and length $e$. Write it as
\begin{align*}
p: u = u_1\xrightarrow{(\alpha_1,m_1)}\cdots\xrightarrow{(\alpha_e,m_e)} u_{e+1} = v.
\end{align*}
\begin{enumerate}[(a)]
\item Using the triangle inequality for the quantum Bruhat graph \cite[Lemma~1]{Postnikov2005}, we get
\begin{align*}
\wt(u\Rightarrow v)\leq& \wt(u_1\Rightarrow u_2)+\cdots + \wt(u_{e}\Rightarrow u_{e+1}).
\intertext{At each step, apply \cite[Lemma~4.7]{Schremmer2022_newton} to get}
\cdots\leq&\Phi^+(-u_1\alpha_1)\alpha_1^\vee+\cdots + \Phi^+(-u_e \alpha_e)\alpha_e^\vee
\\\leq&m_1\alpha_1^\vee+\cdots + m_e\alpha_e^\vee = \omega.
\end{align*}
\item For each edge $u_i\xrightarrow{(\alpha_i,m_i)}u_{i+1}$, we estimate
\begin{align*}
\langle m_i\alpha_i^\vee,2\rho\rangle = m_i\langle \alpha_i^\vee,2\rho\rangle \geq \Phi^+(-u_i\alpha_i)\langle \alpha_i^\vee,2\rho\rangle.
\end{align*}
If $u_i\alpha_i$ is a positive root, we evaluate $\cdots = 0 \geq \ell(u_{i})-\ell(u_{i+1})+1$. If $u_i\alpha_i$ is a negative root, we apply \cite[Lemma~4.3]{Brenti1998} to see $\langle \alpha_i^\vee,2\rho\rangle\geq \ell(u_i) - \ell(u_{i+1})+1$. In any case, we get
\begin{align*}
\langle m_i\alpha_i^\vee,2\rho\rangle \geq \ell(u_i) - \ell(u_{i+1})+1
\end{align*}
with equality holding if and only if there is an edge $u_i\rightarrow u_{i+1}$ in the quantum Bruhat graph of weight $m_i\alpha_i^\vee$.

Iterating this over all edges of the path $p$, we get the desired inequality.

If equality holds, we get a path in the quantum Bruhat graph
\begin{align*}
p' : u_1\rightarrow \cdots\rightarrow u_e
\end{align*}
of length $e$ and weight $\omega$. Moreover, $p'$ is increasing with respect to our reflection order. By \cite[Theorem~6.4]{Brenti1998}, $p'$ must be a shortest path, so $e = d(u\Rightarrow v)$ and $\omega = \wt(u\Rightarrow v)$.
\item In view of the proof of (b), we have to see that there exists a unique shortest path in the quantum Bruhat graph which is increasing for $\prec$. This statement is found again in \cite[Theorem~6.4]{Brenti1998}.\qedhere
\end{enumerate}
\end{proof}
As an application of our findings so far, we are able to prove the following identity.
\begin{proposition}\label{prop:sioIntersections}
Let $x\in\widetilde W$ and $u\in W$. Then
\begin{align*}
\Bigl(\prescript u{}{}U(L) xI\Bigr)\cap \Bigl(\prescript v{}{}U(L) xI\Bigr) = \Bigl(\prescript u{}{}U(L)\cap\prescript v{}{}U(L)\Bigr)xI.
\end{align*}
\end{proposition}
\begin{proof}
The group $H := \prescript u{}{}U(L)\cap\prescript v{}{}U(L)$ acts on both sides of this equation by left multiplication. Moreover, each $H$-orbit in $\prescript u{}{}U(L)$ contains an element of $\prescript u{}{}U(L)\cap \prescript {vw_0}{}{}U(L)$. We see that each $H$-orbit of 
\begin{align*}
(\prescript u{}{}U(L) xI)\cap (\prescript v{}{}U(L) xI)
\end{align*}
contains an element of the intersection
\begin{align*}
Y := \Bigl(\prescript u{}{}U(L)\cap \prescript {vw_0}{}{}U(L)\Bigr) xI\cap (\prescript v{}{}U(L) xI).
\end{align*}
We claim that $Y = xI$. Indeed, $Y/I\subseteq G(L)/I$ is decomposed into pieces according to the multiset
\begin{align*}
\wts(w_x^{-1} v\Rightarrow w_x^{-1} v\dashrightarrow w_x^{-1} u).
\end{align*}
By Lemma~\ref{lem:qbgNonEmptiness}, we see that this multiset only contains tuples of the form $(\omega,0)$, i.e.\ all elements in there correspond to length zero paths from $w_x^{-1} v$ to $w_x^{-1}v$ in the double Bruhat graph. Since there is exactly one such path, we conclude that $Y/I$ is irreducible of dimension zero, that is, a point. So the inclusion $xI\subseteq Y$ is an equality.

We summarize that
\begin{align*}
&(\prescript u{}{}U(L) xI)\cap (\prescript v{}{}U(L) xI) = HY = HxI = \Bigl(\prescript u{}{}U(L)\cap\prescript v{}{}U(L)\Bigr)xI.\qedhere
\end{align*}
\end{proof}
\section{Affine Deligne--Lusztig varieties}\label{sec:ADLV}
In this section, we study affine Deligne--Lusztig varieties in the affine flag variety. So the parahoric subgroup is $I$ and the variety $X_x(b)$ is parametrized by $x\in\widetilde W$ and $[b]\in B(G)$. In addition to the restriction on split groups over a local field of equal characteristic, we also fix a $\sigma$-conjugacy class $[b]\in B(G)$ whose Newton point $\nu(b)\in X_\ast(T)\otimes\mathbb Q$ is \emph{integral}, i.e.\ contained in $X_\ast(T)$. Then $b=t^{\nu(b)}\in\widetilde W$ is our canonical representative of $[b]\in B(G)$.

Our three main questions regarding the geometry of $X_x(b)$ are answered by Görtz-Haines-Kottwitz-Reumann in terms of intersections of semi-infinite orbits with affine Schubert cells. Due to different choices of Iwahori subgroups, we have a few signs different from the original source. We define the dimension of the empty variety to be $-\infty$.

\begin{theorem}[{\cite[Theorem~6.3.1]{Goertz2006}}]\label{thm:ghkr}
Let $x,z\in \widetilde W$. Then
\begin{align*}
\dim \left(X_x(b)\cap U(L)z^{-1}I/I\right) = \dim\left(IxI/I~\cap~\Bigl(\prescript{z}{}{}U(L)\Bigr)z b z^{-1}I/I\right).
\end{align*}
The number of $(J_b(F)\cap U(L))$-orbits of top dimensional irreducible components in $X_x(b)\cap U(L)z^{-1}I/I$ is equal to the number of top dimensional irreducible components in \begin{align*}&IxI/I~\cap~\Bigl(\prescript{z}{}{}U(L)\Bigr)z b z^{-1}I/I.\rightqed\end{align*}
\end{theorem}
The statement on irreducible components is missing in the cited source, but follows since the proof method allows to compare admissible subsets in $X_x(b)\cap U(L)z^{-1}I/I$ with admissible subsets in the other intersection. A generalization of Theorem~\ref{thm:ghkr} to non-integral $[b]$ can be found in \cite[Theorem~11.3.1]{Goertz2010}, but it is unclear how our methods can be applied to that generalized statement, or how a connection to the double Bruhat graph would be given in general.

If we write $z = u t^{\mu_z}$, then
\begin{align*}
\Bigl(\prescript z{}{}U(L)\Bigr)zb z^{-1} I = \prescript u{}{}U(L)t^{u\nu(b)}I.
\end{align*}
Our main result describing the intersection of that set with $IxI$ is the following.
\begin{theorem}\label{thm:generalizedMV}
Let $x=w_xt^{\mu_x}, y = w_yt^{\mu_y}\in\widetilde W$. Let $v_x\in\LP(x)$ and $u\in W$. Pick a reflection order $\prec$ and an index $n\in\{0,\dotsc\#\Phi^+\}$ such that $\pi_{\succ n} = u^{-1}w_x v_x w_0$. Then
\begin{align*}
(IxI/I)\cap (\prescript {uw_0}{}{}U(L)yI/I) = \bigsqcup_{p\in P}\tilde {\mathcal T}_p\subset G(L)/I,
\end{align*}
where
\begin{align*}
P = \{p\in \paths^\prec_{\preceq n}(w_y^{-1} u \Rightarrow w_x^{-1} u)\mid \wt(p) = u^{-1} w_y \mu_y - u^{-1} w_x \mu_x\}
\end{align*}
and each $\tilde {\mathcal T}_p\subset G(L)/I$ is a locally closed $k$-subscheme of finite dimension
\begin{align*}
\dim \tilde {\mathcal T}_p = \frac 12\left(\ell(x)-\ell_u(y) + \ell(p)\right) - \codim(\mathcal T_p \cap (x^{-1}IxI)~\subseteq~ \mathcal T_p).
\end{align*}
Here, $\mathcal T_p\subset G(L)/I$ is the variety from Theorem~\ref{thm:semiInfiniteOrbitsViaPaths}. If $\mathcal T_p\cap x^{-1}IxI$ is dense in $\mathcal T_p$, then $\tilde {\mathcal T}_p$ is irreducible.
\end{theorem}
\begin{proof}
We may write $IxI/I = I(x)xI/I\cong I(x)$ where $I(x)$ is the finite dimensional $k$-group
\begin{align*}
I(x) = \prod_a U_a,
\end{align*}
with the product taken over all positive affine roots $a\in\Phi_\af^+$ such that $x^{-1}a\in \Phi_\af^-$. If $a = (\alpha,m)$ is such a root, recall that the root subgroup $U_a$ is defined as the subvariety $U_a = \{U_\alpha(h t^m)\mid h\in k\}$. The length positivity of $v_x$ implies $(w_x v_x w_0)^{-1}\alpha\in \Phi^+$. Thus we can rewrite the condition $v_x\in \LP(x)$ as $I(x)\subset \prescript{w_x v_x w_0}{}{}U(L)$, or
\begin{align*}
IxI = (I\cap \prescript{w_x v_x w_0}{}U(L))xI.
\end{align*}
Observe that we have an isomorphism of $k$-ind-schemes
\begin{align}
&\Bigl(\prescript{w_x v_x w_0}{}{}U(L)\cap \prescript u{}{}U(L)\Bigr)\times \Bigl(\prescript{w_x v_x w_0}{}{}U(L)\cap \prescript {uw_0}{}{}U(L)\Bigr)
\to \prescript{w_x v_x w_0}{}{}U(L),\notag
\\&(g_1,g_1)\mapsto g_1 g_2.\label{eq:Udecomp1}
\end{align}
Such an isomorphism may e.g.\ be obtained by decomposing $U(L)$ into a suitable product of root subgroups $U_\alpha(L)$ as discussed in the beginning of Section~\ref{sec:semiInfiniteOrbits}.

Using \eqref{eq:Udecomp1}, we can write each element $g\in I(x)$ uniquely in the form $g = g_1 g_2$ with $g_1\in I(x)\cap \prescript {uw_0}{}{}U(L)$ and $g_2\in I(x)\cap \prescript{u}{}{}U(L)$. Then $gxI\in \prescript {uw_0}{}{}U(L)yI$ holds if and only if $g_2xI\in\prescript {uw_0}{}{}U(L)yI$. Hence the decomposition \eqref{eq:Udecomp1} yields an isomorphism
\begin{align}
&(IxI\cap \prescript {uw_0}{}{}U(L)yI)/I = (I(x)xI\cap \prescript {uw_0}{}{}U(L)yI)/I\notag
\\\cong& \Bigl((I(x)\cap\prescript {uw_0}{}{}U(L))xI/I\Bigr) \times \Bigl((I(x)\cap \prescript {u}{}{}U(L))xI \cap \prescript {uw_0}{}{}U(L)yI\Bigr)/I.\label{eq:Udecomp2}
\end{align}
The first variety $(I(x)\cap \prescript{uw_0}{}{}U(L))xI/I \cong I(x)\cap \prescript{uw_0}{}{}U(L)$ is just an affine space over $k$ whose dimension is given by the number of positive affine roots $a = (\alpha,m)$ with $x^{-1}a\in\Phi_\af^-$ and $(uw_0)^{-1}\alpha\in \Phi^+$. By \cite[Lemma~2.9]{Schremmer2022_newton}, we can express this quantity as
\begin{align*}
S_1 := \sum_{\alpha\in \Phi^-} \max(0,\ell(x^{-1},u\alpha)).
\end{align*}
Let us moreover define
\begin{align*}
S_2 := \sum_{\alpha\in \Phi^-} \min(0,\ell(x^{-1},u\alpha)).
\end{align*}
Then $-S_1 - S_2$ is simply the sum over all $\ell(x^{-1},-u\alpha)$ for $\alpha\in \Phi^+$, which we denoted by $\ell_u(x)$. Conversely, $S_1-S_2$ is the sum over all $\abs{\ell(x^{-1},u\alpha)}$, which equals $\ell(x)$ by \cite[Corollary~2.10, Lemma~2.6]{Schremmer2022_newton}. We conclude
\begin{align*}
\dim I(x)\cap \prescript u{}{}U(L) = S_1 = \frac 12\left(\ell(x) - \ell_u(x)\right).
\end{align*}
It remains to study the second factor in \eqref{eq:Udecomp2}. Following Theorem~\ref{thm:semiInfiniteOrbitsViaPaths}, we may decompose
\begin{align*}
\Bigl[\prescript{w_x v_x w_0}{}{}U(L)\cap \prescript{u}{}{}U(L)\Bigr]xI= \bigsqcup_p x\mathcal T_p,
\end{align*}
with the union taken over all paths $p\in \paths^{\prec}_{\preceq n}(u'\Rightarrow w_x^{-1} u)$ and all $u'\in W$.
Using Theorem~\ref{thm:semiInfiniteOrbitsViaPaths}, we conclude that $x \mathcal T_p \cap \prescript {uw_0}{}{}U(L)yI$ is empty if $p\notin P$ and equal to $x\mathcal T_p$ if $p\in P$. Hence the second factor in \eqref{eq:Udecomp2} decomposes as
\begin{align*}
\Bigl((I\cap \prescript{w_x v_x w_0}{}{}U(L)\cap \prescript u{}{}U(L))xI\cap \prescript {uw_0}{}{}U(L) yI\Bigr)/I = \bigsqcup_{p\in P}\Bigl( x\mathcal T_p\cap IxI/I\Bigr).
\end{align*}
We have
\begin{align*}
\dim \Bigl( x\mathcal T_p \cap IxI/I\Bigr) = \dim \mathcal T_p - \codim(\mathcal T_p\cap x^{-1} IxI/I~\subseteq~\mathcal  T_p).
\end{align*}
So the piece $\tilde{\mathcal  T}_p$ corresponding to $I(x)\cap \prescript{uw_0}{}{}U(L)$ and $x\mathcal T_p \cap IxI$ under \eqref{eq:Udecomp2} has dimension
\begin{align*}
\frac 12\left(\ell(x)-\ell_u(x)\right) + \frac 12\left(\ell_u(x)-\ell_u(y)+\ell(p)\right) - \codim(\mathcal T_p\cap x^{-1} IxI/I~\subseteq~ \mathcal T_p).
\end{align*}
Cancelling the common term $\ell_u(x)$, we get the claimed formula.
If $\mathcal T_p\cap x^{-1}IxI$ is dense in $\mathcal T_p$, then $\mathcal T_p\cap x^{-1} IxI$ is irreducible itself. Thus $\tilde {\mathcal T}_p$ is isomorphic to the direct product of the affine space $I(x)\cap \prescript{uw_0}{}{}U(L)$ and the irreducible variety $\mathcal T_p\cap x^{-1} IxI$, hence irreducible itself.
\end{proof}
We are especially interested in those situations where all $p\in P$ satisfy the property $\mathcal T_p\subseteq x^{-1} IxI$. This is not guaranteed at all, as one may choose $p$ and $x$ independently to obtain examples where the inclusion is far from being satisfied. Nonetheless, we can develop some regularity conditions imposed on $(x,y)$ that guarantee this inclusion.
\begin{lemma}\label{lem:semiInfiniteRegularity}
Let $\Phi^+ = \{\beta_1\prec\cdots\prec\beta_{\#\Phi^+}\}$ be a reflection order, $u,v\in W$ and $p\in\paths^\prec(u\Rightarrow v)$. Pick $gI\in \mathcal T_p$ and write it as
\begin{align*}
gI = U_{v\beta_1}(g_1)\cdots U_{v\beta_{\#\Phi^+}}(g_{\#\Phi^+})I\in\mathcal  T_p.
\end{align*}
Then for $m=1,\dotsc,\#\Phi^+$, we have
\begin{align*}
\nu_L(g_m) \geq -3\langle \rho^\vee,\beta_m\rangle \langle \wt(p),\rho\rangle,
\end{align*}
where $\rho$ is the half-sum of positive roots and $\rho^\vee$ the half-sum of positive coroots.
\end{lemma}
\begin{proof}
Write our path as
\begin{align*}
p : u = u_1\xrightarrow{(\alpha_1,m_1)}\cdots\xrightarrow{(\alpha_{\ell(p)},m_{\ell(p)})} u_{\ell(p)+1} = v.
\end{align*}
Denote the type corresponding to $p$ under Lemma~\ref{lem:typesAndPaths} by $\tau = \{(n_1,\nu_1),\dotsc,(n_{\ell(p)},\nu_{\ell(p)})\}$. This means $\alpha_i=\beta_{n_i}$ and
\begin{align*}
(s_{\alpha_{\ell(p)}}\cdots s_{\alpha_{i+1}}(\alpha_i),m_i) = r_{(\alpha_{\ell(p)},\nu_{\ell(p)})}\cdots r_{(\alpha_{i+1},\nu_{i+1})}(\alpha_i,\nu_i)\in\Phi_\af.
\end{align*}
We also write $\alpha' := s_{\alpha_{\ell(p)}}\cdots s_{\alpha_{i+1}}(\alpha_i)\in\Phi$.

Let $f^{(\bullet)}_\bullet$ be the polynomials used in the definition of $\mathcal T_p = \mathcal T_{u,\prec,\tau}$, i.e.\ the polynomials from Proposition~\ref{prop:semiInfiniteTypes}.
We write
\begin{align*}
g^{(m)}_i =&g_i + f^{(m)}_i(g_{i+1},\dotsc,g_{\#\Phi^+}),
\\g'_m &= g^{(m)}_m.
\end{align*}

Let $m\in\{1,\dotsc,\#\Phi^+\}$ and let $h = h(m)\in\{0,\dotsc,\ell(p)\}$ be maximal such that $\alpha_{h'}\prec \beta_m$ for all $h'\leq h$. Then the condition $g\in\mathcal  T_p$ implies, by definition, that $g_m'=0$ or
\begin{align*}
r_{(u\alpha_{\ell(p)},\nu_{\ell(p)})}\cdots r_{(u\alpha_{h+1},\nu_{h+1})} (u\beta_m,\nu_L(g'_m))\in \Phi_\af^+.
\end{align*}
Observe that we can write
\begin{align*}
&r_{(u\alpha_{\ell(p)},\nu_{\ell(p)})}\cdots r_{(u\alpha_{h+1},\nu_{h+1})} (u\beta_m,\nu_L(g'_m))
\\=\,&r_{(u\alpha'_{h+1},m_{h+1})}\cdots r_{(u\alpha'_{\ell(p)},m_{\ell(p)})}(u\beta_m,\nu_L(g'_m))
\\=\,&(v\beta_m,\nu_L(g'_m)) + \sum_{i=h+1}^{\ell(p)} c_i (u\alpha_i',m_i)
\end{align*}
for elements $c_i\in\{0,\pm 1,\pm 2,\pm 3\}$ describing the pairings between the occurring roots. Thus
\begin{align*}
\nu_L(g_m')\geq -3\sum_{i=h+1}^{\ell(p)} m_i \geq -3\sum_{i=h+1}^{\ell(p)}m_i\langle \alpha_i^\vee,\rho\rangle \geq -3\langle \wt(p),\rho\rangle.
\end{align*}

We claim for all $m\in\{1,\dotsc,\#\Phi^+\}$ and $i\in\{1,\dotsc,n\}$ that
\begin{align*}
\nu_L(g^{(m)}_i - g_i')\geq -3\langle \rho^\vee,\beta_i\rangle\langle \wt(p),\rho\rangle.
\end{align*}
Induction on $m-i$. For $m=i$, we have $g^{(m)}_i = g'_i$ by definition, so there is nothing to prove.

So let now $i<m$ and suppose the claim has been proved for all pairs of smaller difference.

If $\beta_{m}\notin\{\alpha_1,\dotsc,\alpha_{\ell(p)}\}$, then applying Lemma~\ref{lem:semiInfiniteConjugation} (which is how the polynomials $f^{(\bullet)}_\bullet$ were constructed) yields $g^{(m)}_i = g^{(m-1)}_i$, so we are done by induction immediately.

So suppose now that $\beta_{m}\in\{\alpha_1,\dotsc,\alpha_{\ell(p)}\}$. Applying Lemma~\ref{lem:semiInfiniteConjugation}, we see that $g^{(m-1)}_i$ has the form
\begin{align*}
g^{(m-1)}_i = g^{(m)}_i + \sum c_{e_{i+1},\dotsc,e_{m}} (g^{(m)}_{i+1})^{e_{i+1}}\cdots (g^{(m)}_{m})^{e_{m}},
\end{align*}
with the sum taken over all possible integers $e_{i+1},\dotsc,e_{m-1}\geq 0, e_{m}<0$ such that $\beta_i = e_{i+1}\beta_{i+1}+\cdots+e_{m} \beta_{m}$ and structure constants $c_{e_{i+1},\dotsc,e_{m}}\in\mathbb Z$. By the inductive assumption and the above estimate on $\nu_L(g_\bullet')$, we see
\begin{align*}
\nu_L\Bigl[(g^{(m)}_{i+1})^{e_{i+1}}\cdots (g^{(m)}_{n-1})^{e_{n-1}}\Bigr] \geq -3\langle \wt(p),\rho\rangle\langle \rho^\vee,e_{i+1}\beta_{i+1}+\cdots+e_{i_{m-1}} \beta_{m-1}\rangle.
\end{align*}
An entirely similar argument to the one presented above shows moreover $\nu_L(g_m') \leq 3\langle \wt(p),2\rho\rangle$, so that
\begin{align*}
\nu_L((g_m')^{e_m}) \geq -3 e_m \langle \wt(p),\rho\rangle\langle \rho^\vee,\beta_m\rangle.
\end{align*}
We conclude
\begin{align*}
\nu_L\Bigl[c_{e_{i+1},\dotsc,e_{m}} (g^{(m)}_{i+1})^{e_{i+1}}\cdots (g^{(m)}_{m})^{e_{m}}\Bigr] &\geq -3\langle \wt(p),\rho\rangle \langle \rho^\vee,e_{i+1}\beta_{i+1}+\cdots + e_m\beta_m\rangle 
\\&=-3\langle \wt(p),\rho\rangle \langle \rho^\vee,\beta_i\rangle.
\end{align*}
Hence
\begin{align*}
\nu_L(g^{(m)}_i - g^{(m-1)}_i) \geq -3\langle \rho^\vee,\beta_i\rangle \langle \wt(p),\rho\rangle.
\end{align*}
This finishes the induction.

In particular, we see
\begin{align*}
\nu_L(g_m)\geq -3\langle \rho^\vee,\beta_m\rangle \langle \wt(p),\rho\rangle
\end{align*}
for all $m$.
This finishes the proof.
\end{proof}
We finally define the class of elements in $\widetilde W$ where Theorem~\ref{thm:generalizedMV} and Lemma~\ref{lem:semiInfiniteRegularity} describe affine Deligne--Lusztig varieties fully.
\begin{definition}\label{def:superparabolic}
Let $x=wt^\mu\in\widetilde W, J\subseteq\Delta$ and $C\in\mathbb R_{>0}$. We say that $x$ is \emph{$(J,C)$-superparabolic} if there exists $v\in W$ such that
\begin{enumerate}[(a)]
\item all $\alpha\in\Phi_J$ satisfy $\ell(x,v\alpha)=0$ and
\item all $\alpha\in\Phi^+\setminus\Phi_J$ and $v'\in vW_J$ satisfy \begin{align*}\langle \mu,v'\alpha\rangle> C\langle \rho^\vee,\alpha\rangle
\end{align*}
\end{enumerate}
\end{definition}
This is a generalization of the $J$-adjusted and $J$-superdominant elements from \cite{Lenart2015}. If $x$ is $(J,2)$-superparabolic and $v$ as in the above definition, then one easily checks $\LP(x) = vW_J$. We can interpret condition (b) above as a regularity condition of the length functional, in particular
\begin{align*}
&\Bigl(\forall \alpha\in \Phi^+\setminus \Phi_J:~\ell(x,v\alpha) > 1+C\langle \rho^\vee,2\rho\rangle\Bigr)
\\\implies&\text{condition (b) of Definition~\ref{def:superparabolic}}
\\\implies&\Bigl(\forall \alpha\in\Phi^+\setminus\Phi_J:~\ell(x,v\alpha)>C-1\Bigr).
\end{align*}
\begin{theorem}\label{thm:adlvViaSemiInfiniteOrbits}
Let $x=wt^\mu\in\widetilde W$ and $b = t^{\nu(b)}$ be an element with integral dominant Newton point $\nu(b)\in X_\ast(T)^{\dom}$. Define for each $v\in \LP(x)$ and $u\in W$ the multiset
\begin{align*}
E(u,v) := \{e \mid (u^{-1} \mu-\nu(b),e)\in \wts(u\Rightarrow wu\dashrightarrow wv)\}_m.
\end{align*}
Put $\max\emptyset :=-\infty$ and define
\begin{align*}
e :&= \max_{u\in W} \min_{v\in \LP(x)} \max(E(u,v))\in\mathbb Z\cup \{-\infty\}.\\
d :&= \frac 12\left(\ell(x)+e-\langle \nu(b),2\rho\rangle\right)\in\mathbb Z\cup\{-\infty\}.
\end{align*}
\begin{enumerate}[(a)]
\item If there exists for every $u\in W$ some $v\in \LP(x)$ with $E(u,v)=\emptyset$, i.e.\ if $e=d=-\infty$, then $X_x(b)=\emptyset$.
\item If $X_x(b)\neq\emptyset$, then $\dim X_x(b)\leq d$.
\item 
Write $C = 3\langle \mu^\dom-\nu(b),\rho\rangle$ and suppose that $x$ is $(J,C)$-superparabolic for some $J\subseteq \Delta$. Define the multiset $E$ as the additive union
\begin{align*}
E = \bigcup_{v\in\LP(x)} E(vw_0(J),v).
\end{align*}
Then $X_x(b)\neq \emptyset$ if and only if $E\neq \emptyset$. In this case, $e=\max(E)$ and $\dim X_x(b)=d$.
\item Assume $X_x(b)\neq\emptyset$ and let $\Sigma_d$ be the set of $d$-dimensional irreducible components of $X_x(b)$.

Then the number of $J_b(F)$-orbits in $\Sigma_d$ is
\begin{align*}
\#(\Sigma_d/J_b(F))\leq \sum_{u\in W} \min_{v\in W}\left(\text{multiplicity of $e$ in $E(u,v)$}\right).
\end{align*}
If we are in the situation of (c) and $b$ is regular, i.e.\ $\langle \nu(b),\alpha\rangle\neq 0$ for all $\alpha\in\Phi$, then $\#(\Sigma_d/J_b(F))$ is equal to the multiplicity of $e$ in $E$.
\end{enumerate}
\end{theorem}
\begin{proof}
We use Theorem~\ref{thm:ghkr} to reduce questions on the affine Deligne--Lusztig variety to the situation of Theorem~\ref{thm:generalizedMV}.

So let $z = ut^{\mu_z}\in\widetilde W$. Then $X_x(b)\cap U(L)z^{-1}I/I$ is closely related to the intersection
\begin{align*}
(IxI/I)\cap (\prescript u{}{}U(L)t^{-u\nu(b)}I/I).
\end{align*}
Pick $v\in \LP(x)$.
By Theorem~\ref{thm:generalizedMV}, the latter intersection can be decomposed into pieces $(\tilde {\mathcal T}_p)_{p\in P}$ with
\begin{align*}
P = \{p\in\paths^\prec_{\preceq n}(u w_0\Rightarrow w^{-1} uw_0)\mid \wt(p) = (uw_0)^{-1} (u\nu(b)) - (uw_0)^{-1} w\mu\}.
\end{align*}
Here, $\prec$ is a reflection order chosen such that $\pi_{\succ n} = w_0u^{-1} w vw_0$. The number of paths in $P$ having a given length $\ell\in\mathbb Z_{\geq 0}$ is equal to the multiplicity of
\begin{align*}
(w_0(\nu(b) - u^{-1} w\mu),\ell)
\end{align*}
in the multiset
\begin{align*}
\wts(u w_0\Rightarrow w^{-1} uw_0\dashrightarrow vw_0).
\end{align*}
By Lemma~\ref{lem:pathSymmetries} (a) and (b), this is also equal to the multiplicity of $(u^{-1} w \mu-\nu(b),\ell)$ in
\begin{align*}
&\wts(w_0 u\Rightarrow w_0w^{-1} u\dashrightarrow w_0 v)
\\=\,&\wts(w^{-1} u\Rightarrow u \dashrightarrow wv).
\end{align*}
By definition, this is the multiplicity of $\ell$ in the multiset $E(w^{-1}u,v)$.

We see that if $E(w^{-1} u,v)=\emptyset$ then $X_x(b)\cap U(L)z^{-1} I=\emptyset$. Otherwise,
\begin{align*}
\dim \Bigl(X_x(b)\cap U(L)z^{-1} I/I\Bigr) = \max_p \dim \tilde{\mathcal  T}_p \leq \frac 12\left(\ell(x)+\max(E(w^{-1} u,v))-\ell_{uw_0}(t^{u\nu(b)})\right).
\end{align*}
Observe $\ell_{uw_0}(t^{u\nu(b)}) = \langle \nu(b),2\rho\rangle$. This shows (a) and (b).

In particular, we have $\dim \tilde {\mathcal T}_p \leq d$ for all $p$. If equality holds, then $\tilde {\mathcal T}_p$ must be irreducible by Theorem~\ref{thm:generalizedMV}. Hence the number of $d$-dimensional irreducible components in $X_x(b)\cap U(L)z^{-1} I/I$ is equal to the number of pieces $\tilde {\mathcal T}_p$ satisfying $\dim \tilde {\mathcal T}_p=d$, which is at most the multiplicity of $e$ in $E(w^{-1} u,v)$. Observe that the action of $T(F)\subseteq J_B(F)$ simply permutes the intersections $X_x(b)\cap U(L)z^{-1} I/I$ by changing the value of $\mu_z$. Thus the number $\#(\Sigma_d/J_b(F))$ is at most equal to
\begin{align*}
\sum_{u\in W}(\text{number of $d$-dimensional irreducible components in }(X_x(b)\cap U(L)u^{-1}I)).
\end{align*}
We get the desired estimate in (d).

Let us now assume the regularity condition from (c). Let $v_1\in W$ be chosen such that $v_1^{-1}\mu$ is dominant. We have $\LP(x) = v_1 W_J$.

If $u\in W$ satisfies $w^{-1} u\notin \LP(x) = v_1 W_J$, we find a positive root $\alpha\in \Phi^+\setminus \Phi_J$ with $v_1^{-1} w^{-1} u\alpha\in \Phi^-$. Hence
\begin{align*}
\langle \mu,-w^{-1}u \alpha\rangle \geq C = 3\langle v_1^{-1}\mu-\nu(b),\rho\rangle.
\end{align*}
We conclude
\begin{align*}
w^{-1}u \mu \leq \mu^{\dom} - C\alpha^\vee\not\geq \nu(b).
\end{align*}
Hence $E(u,v)=\emptyset$ for all $v\in \LP(x)$, proving $IxI\cap \prescript u{}{}U(L)t^{-u\nu(b)}I=\emptyset$.

Let us now consider the case $w^{-1} u \in \LP(x)$. Then also $v := w^{-1} u w_0(J)\in \LP(x)$. Consider the pieces $\tilde {\mathcal T}_p$ as constructed above for this pair $(u,v)$, i.e.\ for paths $p$ from $uw_0$ to $w^{-1} uw_0$. We claim $\mathcal T_p\subseteq x^{-1} Ix I$ for all occurring paths $p$, using Lemma~\ref{lem:semiInfiniteRegularity}: Indeed if
\begin{align*}
gI = U_{w^{-1}uw_0\beta_1}(g_1)\cdots U_{w^{-1}uw_0\beta_{n}}(g_{n})I\in \mathcal T_p
\end{align*}
and $w_0 w_0(J) w_0 = \pi_{\succ\beta_n}$, then $-w_0\Phi^+_{J} = \{\beta_{n+1},\dotsc,\beta_{\#\Phi^+}\}$. The condition that $p$ is bounded above by $n$ yields
\begin{align*}
w^{-1} uw_0\beta_1,\dotsc, w^{-1} u w_0\beta_n \in w^{-1} u (\Phi^-\setminus \Phi_J).
\end{align*}
Since $w^{-1} u\in \LP(x)$, the superparabolicity condition implies $\langle \mu,w^{-1} u w_0\beta_i\rangle < -C\langle \rho^\vee,\beta_i\rangle$. By Lemma~\ref{lem:semiInfiniteRegularity}, we obtain $x U_{w^{-1} uw_0\beta_i}(g_i)x^{-1} \in I$. This shows the claim $\mathcal T_p\subseteq x^{-1}IxI$.

By Theorem~\ref{thm:generalizedMV}, we see that $\tilde {\mathcal T}_p$ is irreducible of dimension
\begin{align*}
\dim \tilde {\mathcal T}_p = \frac 12\left(\ell(x)+\ell(p)-\langle \nu(b),2\rho\rangle\right).
\end{align*}
This fully describes non-emptiness, dimension and top dimensional irreducible components of $X_x(b)\cap U(L)z^{-1} I/I$. So $X_x(b)\neq\emptyset$ if and only if $E\neq\emptyset$, and in this case
\begin{align*}
\dim X_x(b) = \frac 12\left(\ell(x)+\max(E)-\langle \nu(b),2\rho\rangle\right).
\end{align*}
We saw $E(u,v)=\emptyset$ whenever $u\notin \LP(x)$, so $e\leq \max(E)$ by definition of $e$. Conversely, we get $\max(E)\leq e$ from (b) and the above dimension calculation. Thus $\max(E)=e$. We conclude (c).

Assume now that $[b]$ is regular as in the final claim of (d). Then $J_b(F) = T(F)$, so the number of $J_b(F)$-orbits of $d$-dimensional irreducible components in $X_x(b)$ is equal to
\begin{align*}
\sum_{u\in W}(\text{number of $d$-dimensional irreducible components in }(X_x(b)\cap U(L)u^{-1}I)).
\end{align*}
Observe that each summand is equal to be the number of $d$-dimensional pieces $\tilde{\mathcal  T}_p$ corresponding to $u\in W$. By the above analysis using the superparabolicity assumption, this number is equal to the multiplicity of $e$ in $E(w^{-1} u, w^{-1} u w_0(J))$ (thus zero if $w^{-1} u\notin \LP(x)$). The final claim of (d) follows.
\end{proof}
\begin{remark}
If $x$ is not superparabolic, we do not expect that the converse of Theorem~\ref{thm:adlvViaSemiInfiniteOrbits} (a) holds in general. Even if $X_x(b)\neq\emptyset$, we do not expect that equality holds in (b) or (d) in general. It is easy to find counterexamples using a computer search.
\end{remark}
\begin{corollary}\label{cor:superregularADLV}
Let $x = wt^\mu\in \widetilde W$ and $[b]\in B(G)$. Let $v\in W$ such that $v^{-1}\mu$ is dominant. Put $C = 3\langle v^{-1}\mu-\nu(b),\rho\rangle$ and assume that
\begin{align*}
\langle v^{-1}\mu,\alpha\rangle \geq C
\end{align*}
for all simple roots $\alpha$. Define the multiset
\begin{align*}
E = \{e\in\mathbb Z\mid (v^{-1}\mu-\nu(b),e)\in \wts(v\Rightarrow wv)\}_m.
\end{align*}
Then $X_x(b)\neq\emptyset$ if and only if $E\neq\emptyset$. In this case, the dimension of $X_x(b)$ is
\begin{align*}
\dim X_x(b) = \frac 12\left(\ell(x)+\max(E)-\langle \nu(b),2\rho\rangle\right),
\end{align*}
and the number of $J_b(F)$-orbits of top dimensional irreducible components is equal to the multiplicity of $\max(E)$ in $E$.
\end{corollary}
\begin{proof}
The regularity condition on $(x,b)$ implies that $\nu(b)$ must be regular. In particular, $[b]$ is integral. Now apply the previous theorem.
\end{proof}
\begin{remark}
We saw in Proposition~\ref{prop:qbgVsDbg} that the set $\{\omega\mid (\omega,e)\in\wts(v\Rightarrow wv)\}$ contains a unique minimum, which is given by the weight of a shortest path in the quantum Bruhat graph from $v$ to $wv$. The above corollary shows under some strong regularity conditions that the set $B(G)_x$ contains a unique maximum $[b_x]$, being the element of Newton point $\nu(b_x)=v^{-1}\mu-\wt(v\Rightarrow wv)$. It moreover follows that this element $[b_x]$ satisfies
\begin{align*}
d(v\Rightarrow wv) = \ell(x)-\langle \nu(b_x),2\rho\rangle = \dim X_x(b_x)
\end{align*} and that $X_x(b_x)$ has, up to $J_{b_x}(F)$-action, only one top dimensional irreducible component.

It is a well-known result of Viehmann that $B(G)_x$ always contains a unique maximum for arbitrary $G$ and $x$ \cite[Section~5]{Viehmann2014}. She moreover provides a combinatorial description in terms of the Bruhat order on the extended affine Weyl group.

The description of the generic $\sigma$-conjugacy class $[b_x]$ in terms of the quantum Bruhat graph is known due to Mili\'cevi\'c \cite{Milicevic2021}. Her proof is more combinatorial in nature, following Viehmann's description of $[b_x]$ using the Bruhat order and a comparison of the Bruhat order with the quantum Bruhat graph due to Lam-Shimozono \cite{Lam2010}. Her combinatorial methods have been refined since, so that a description of $[b_x]$ using the quantum Bruhat graph is known for arbitrary $G$ and $x$ \cite{Sadhukhan2022_qbg, He2024_demazure, Schremmer2022_newton}. Our corollary recovers Mili\'cevi\'c's original result using an entirely different proof method, which moreover reveals how to find the quantum Bruhat graph itself in the affine flag variety.

The aforementioned geometric properties of $X_x(b_x)$ are well-known for arbitrary $G$ and $x$, as described in the introduction. We can interpret Proposition~\ref{prop:qbgVsDbg}, i.e.\ essentially \cite[Theorem~6.4]{Brenti1998}, as a combinatorial shadow of these geometric facts.
\end{remark}
\begin{remark}\label{rem:chenZhu}
Let us compare Corollary~\ref{cor:superregularADLV} to the situation of affine Deligne--Lusztig varieties in the affine Grassmannian, i.e.\ where the parahoric subgroup $K = G(\mathcal O_L)$ is hyperspecial. Given $[b]\in B(G)$ and a dominant $\mu\in X_\ast(T)$, we can compare the affine Deligne--Lusztig variety $X_\mu(b)\subset G(L)/K$ with $X_{w_0t^\mu}(b)\subset G(L)/I$ following \cite[Theorem~10.1]{He2014}. If $x = w_0t^\mu$ satisfies the regularity conditions from Corollary~\ref{cor:superregularADLV}, this means we should study the multiset $\wts(1\Rightarrow w_0)$.

Given any reflection order $\Phi^+ = \{\beta_1\prec\cdots\prec\beta_{\#\Phi^+}\}$, there exists a unique unlabelled path of maximal length from $1$ to $w_0$, given by
\begin{align*}
\overline p: 1\xrightarrow{\beta_1} s_{\beta_1}\xrightarrow{\beta_2}\cdots\xrightarrow{\beta_{\#\Phi^+}} w_0.
\end{align*}
Each arrow in this path is increasing the length in $W$. Thus, the labelled paths $p$ in the double Bruhat graph with underlying unlabelled path $\overline p$ are precisely the paths of the form
\begin{align*}
p: 1\xrightarrow{(\beta_1,m_1)}s_{\beta_1}\xrightarrow{(\beta_2,m_2)}\cdots\xrightarrow{(\beta_{\#\Phi^+},m_{\#\Phi^+})} w_0
\end{align*}
for integers $m_1,\dotsc,m_{\#\Phi^+}\geq 0$. In the situation of Corollary~\ref{cor:superregularADLV}, we see that $X_x(b)\neq\emptyset$ if and only if $\mu-\nu(b)$ is a sum of positive coroots, in which case we get
\begin{align*}
\dim X_x(b) = \frac 12\left(\ell(x)+{\#\Phi^+}-\langle \nu(b),2\rho\rangle\right).
\end{align*}
The number of top dimensional irreducible components of $X_x(b)$ is equal to the number of different ways to express $\mu-\nu(b)$ as a sum of positive coroots. This latter quantity is known as \emph{Kostant's partition function}, which is also known to describe the dimension of the $\nu(b)$-weight space associated with the Verma module $V_\mu$. Under the regularity assumptioned made, this is also equal to the dimension of the $\nu(b)$-weight space of the irreducible quotient $M_\mu$ by Kostant's multiplicity formula.

In view of \cite[Theorem~10.1]{He2014}, we recover Theorem~\ref{thm:hyperspecial} in the setting of Corollary~\ref{cor:superregularADLV}. While this is a fairly restrictive setting, one may expect that statements similar to Corollary~\ref{cor:superregularADLV} hold true in much higher generality.
\end{remark}

\begin{CJK*}{UTF8}{gbsn}
\printbibliography
\end{CJK*}
\end{document}